\documentclass[11pt]{amsart}
\usepackage[numbers]{natbib}
\usepackage{graphicx,amsmath,epsfig,amssymb,amsfonts,amsthm,enumerate,verbatim,todonotes,hyperref,mathtools, dsfont}
\usepackage{amsbsy,mathrsfs}
\usepackage[T1]{fontenc}

\usepackage{tikz,tikz-cd}
\usetikzlibrary{matrix,arrows}

\newtheorem*{rep@theorem}{\rep@title}
\newcommand{\newreptheorem}[2]{%
\newenvironment{rep#1}[1]{%
 \def\rep@title{#2 \ref{##1}}%
 \begin{rep@theorem}}%
 {\end{rep@theorem}}}
\makeatother

\newtheorem{theorem}{Theorem}[section]
\newtheorem{corollary}[theorem]{Corollary}
\newtheorem{lemma}[theorem]{Lemma}
\newtheorem{propo}[theorem]{Proposition}

\newtheorem{question}[theorem]{Question}

\newtheorem{examp}[theorem]{Example}

\theoremstyle{definition}
\newtheorem{definition}[theorem]{Definition}

\newreptheorem{theorem}{Theorem}
\newreptheorem{lemma}{Lemma}
\newreptheorem{definition}{Definition}
\newreptheorem{propo}{Proposition}
\newreptheorem{examp}{Example}
\newreptheorem{notat}{Notation}
\newreptheorem{corollary}{Corollary}

\usepackage{enumitem}
\setenumerate[1]{label=\arabic*)}

\newcommand{\restrict}{\mathord{\upharpoonright}}

\title{A New Topological Generalization of Descriptive Set Theory}
\author[Iv\'an Ongay-Valverde]{Iv\'an Ongay-Valverde$^{1}$}
\author[Franklin D. Tall]{Franklin D. Tall$^{1}$}
\thanks{$^{1}$Supported by NSERC grant A-7354.\\
2022 AMS classification Primary 54H05, 03E15; Secondary 03E15,
03E55, 03E60, 54A35, 54C10, 54C60, 54D40.\\
\textit{K-analytic, generalized descriptive set theory, Menger, Hurewicz, $\sigma$-compact, $\sigma$-projective determinacy, upper semi-continuous compact-valued multifunction.}}
\date{First draft: 2021-12-01 \\ Current Draft: \today }
\begin{document}

\begin{abstract}
We introduce a new topological generalization of the $\sigma$-projective hierarchy, not
limited to Polish spaces. Earlier attempts have replaced $^{\omega}\omega$ by $^{\kappa}\kappa$, for $\kappa$ regular uncountable, or
replaced \emph{countable} by \emph{$\sigma$-discrete}. Instead we close the usual $\sigma$-projective sets
under continuous images and perfect preimages together with countable unions. The natural set-theoretic axiom to apply is \emph{$\sigma$-projective
determinacy}, which follows from large cardinals. Our goal is to generalize the known results for
$K$-analytic spaces (continuous images of perfect preimages
of $^{\omega}\omega$) to these more general settings. We have achieved some successes in the
area of \emph{Selection Principles} \textemdash the general theme is that nicely defined Menger spaces are
Hurewicz or even $\sigma$-compact. The $K$-analytic results are true in ZFC; the more
general results have consistency strength of only an inaccessible.    
\end{abstract}

\maketitle

\section{Introduction}\label{Section intro}

Classical descriptive set theory studies sets of reals, or, more generally, subsets of a Polish (i.e. separable completely metrizable) space, which are in some sense definable \textemdash think of Borel, analytic, etc. See \cite{Kec95}. Due to its intrinsic interest and its usefulness in many areas of mathematics, there have been a number of attempts to generalize it, e.g. Choquet \cite{Ch1959}, Frol\'ik \cite{Fro63}, Sion \cite{Sion}, Stone \cite{Stone1963}, etc. Choquet's attempt retained the central role of the space $\mathds{P}$ of irrationals, but weakened continuous maps to \emph{upper semi-continuous compact-valued} ones (defined later). Frol\'ik, instead, retained continuity but weakened $\mathds{P}$ to any Lindel\"of \v{C}ech-complete space. Jayne (see \cite{RJ80}) observed that, in fact, these two approaches were equivalent, yielding the $K$-analytic spaces:

\begin{definition}\label{def:K-analytic}
\cite{RJ80} A space is \emph{$K$-analytic} (we say \emph{$K$-$\mathbf{\Sigma}^1_1$}) if it is the continuous image of a Lindel\"of \v{C}ech-complete space.
\end{definition}

The $K$-analytic spaces have been applied widely in functional analysis \textemdash see \cite{KKP}. 

Stone and others went in a different direction, generalizing ``countable'' to ``$\sigma$-discrete'', while retaining metrizability. This effort was less successful. Recently, set theorists have generalized the role of $\mathds{P}$ in its equivalent $ ^{\omega}\omega$ form to $ ^{\kappa}\kappa$, for an arbitrary regular cardinal $\kappa$. Although this leads to interesting set theory, it is unlikely to find applications outside of logic. But set theorists have also generalized classical descriptive set theory to consider ``definable'' sets of reals more complicated than Borel or analytic ones. The \emph{$\sigma$-projective sets} are obtained by extending the Borel sets by using continuous real-valued functions, countable unions, and complementation within $\mathds{R}$. Here is the definition. 

\begin{definition}
\cite{Kec95} A subset of reals is \emph{analytic} ($\mathbf{\Sigma}_1^1$) if it is a continuous image of $\mathds{P}$ (equivalently, of any Borel set). It is \emph{co-analytic ($\mathbf{\Pi}^1_1$)} if its complement is analytic.
\end{definition}

\begin{definition}
\cite{Kec95} In general, a subset of reals is $\mathbf{\Pi}^1_{\xi}$, for $\xi<\omega_1$, if its complement is $\mathbf{\Sigma}^1_{\xi}$ and a subset of reals is \emph{$\mathbf{\Sigma}^1_{\xi+1}$} if it is the continuous image of a $\mathbf{\Pi}^1_{\xi}$ set. For $\alpha\leq\omega_1$ a limit ordinal, a subset of reals $X$ is \emph{$\mathbf{\Sigma}^1_{\alpha}$} if there is a $Y_i\in \mathbf{\Sigma}^1_{\xi_i}$, $i<\omega$, $\xi_i<\alpha$ such that $X=\bigcup_{i\in \omega} Y_i$.\\

The class $\mathbf{\Sigma}^1_{\omega}$ is known as the \emph{projective sets} and the class $\mathbf{\Sigma}^1_{\omega_1}$ as the \emph{$\sigma$-projective sets}. Since properties such as \emph{analytic} have been used to refer to subspaces of Polish spaces as well as to general topological spaces, we will use the word \emph{sets} to refer to subspaces of Polish spaces.
\end{definition}

We continue the standard sloppiness of interchangeably speaking of $\mathds{P}$, $\mathds{R}$, $[0, 1]$, and $[0,1]^{\omega}$, since their Borel, projective, etc., structures are isomorphic.

An even more far-reaching extension of the collection of ``definable'' sets of reals is those in $L(\mathds{R})$ (see \cite{Kan94}). It is easy to show that the $\sigma$-projective sets are all in $L(\mathds{R})$. In fact, the $\sigma$-projective sets are precisely those sets of reals which are in $L_{\omega_1}(\mathds{R})$ \cite{AMS21}.

These more extensive collections of definable sets of reals were not easily handled by classical descriptive set theory, but the advent of determinacy axioms consistent with ZFC has made them quite tractable. We will often be assuming such axioms. The novelty of our approach is that we follow Choquet, Frol\'ik, and Rogers-Jayne, but replace $\mathds{P}$ with $\sigma$-projective sets.
Like Frol\'ik, we do not require our continuous maps to be real-valued. There is much to be investigated using a functional analysis lens. We have so far mainly looked at these theories from the viewpoint of \emph{Selection Principles}.

There has been much work in Selection Principles considering the successively stronger Lindel\"of properties: Menger, Hurewicz, $\sigma$-compact. We will give the definitions, but read \cite{TsabanSurvey} for a thorough introduction and survey. 

\begin{definition}
A topological space $X$ is \emph{Menger} if given any countable sequence of open covers $\langle\mathcal{U}_1, \mathcal{U}_2,\ldots, \mathcal{U}_n,\ldots\rangle$ of $X$ there are finite subsets $\mathcal{V}_{i}\subseteq \mathcal{U}_i$ such that $\bigcup_{i\in\omega}\mathcal{V}_{i}$ is also an open cover of $X$.
\end{definition}

\begin{definition}
A topological space $X$ is \emph{Hurewicz} if given any countable sequence of open covers $\langle\mathcal{U}_1, \mathcal{U}_2,\ldots, \mathcal{U}_n,\ldots\rangle$ of $X$, none containing a finite subcover, there are finite subsets $\mathcal{V}_{i}\subseteq \mathcal{U}_i$ such that $\bigcup_{i\in\omega}\mathcal{V}_{i}$ is a \emph{$\gamma$-cover} of $X$. An open cover $\mathcal{O}$ is a \emph{$\gamma$-cover} of $X$ if $\mathcal{O}$ is infinite and every element of $X$ appears in all but finitely members of $\mathcal{O}$.
\end{definition}


(Note Arhangel’ski\u{\i} uses ``Hurewicz'' for the property we now call ``Menger''.)
Menger analytic sets are $\sigma$-compact \cite{Hu}; for $\sigma$-projective sets and, indeed, for sets of reals in $L(\mathds{R})$, it is  consistent from an inaccessible cardinal \cite{Fe}, \cite{DT} and follows from larger cardinals \cite{MS,A} that determinacy axioms hold which imply such Menger sets are $\sigma$-compact.

There has been a series of papers showing that nicely defined Menger sets of reals need not be $\sigma$-compact if $V = L$ is assumed \cite{FM}, but are if determinacy is assumed instead \cite{FM}, and then considering the situation for more general Menger spaces; \cite{TT}, \cite{TTT21}, \cite{Tall20}. Assuming the Axiom of Choice, however, there are Menger sets of reals that are not $\sigma$-compact, indeed not even Hurewicz. See the survey \cite{TsabanSurvey} for references.

One would expect it to be easy to construct a Menger \emph{space} that is not Hurewicz but the second author noticed that, whenever he tried to define a Menger space that was not $\sigma$-compact, it wound up being Hurewicz. He therefore conjectured that “definable” Menger spaces are Hurewicz. Indeed, he proved that $K$-analytic Menger spaces are Hurewicz \cite{Tall20}. However, $K$-analytic Menger spaces need not be $\sigma$-compact \cite{Tall20}.

Just as the $K$-analytic spaces generalize analytic sets, we can define spaces generalizing the $\sigma$-projective sets. Every analytic set is a $K$-analytic space; a $K$-analytic metrizable space is analytic \cite{RJ80}, see Corollary 5.8. (We don't have to require separability, because $K$-analytic spaces are Lindel\"of and Lindel\"of metrizable spaces are separable.) Our generalizations of $\sigma$-projective will have analogous properties.

Continuous functions (maps) are integral to topology; also very important are the \emph{perfect} maps, which are continuous, send closed sets to closed sets, and for which the inverses of points are compact. Many topological properties are preserved by perfect maps and/or their inverses. A topological property is called \emph{perfect} if it is preserved by both perfect maps and their inverses.

\begin{definition}
Given a family $\Gamma$ of subsets of reals, \emph{$\Gamma$ determinacy (Det($\Gamma$))} is the statement, ``given $A\in \gamma$, the perfect information game of countable length and payoff set $A$ has a winning strategy for one of the players''. Then $\sigma$-projective determinacy\textemdash for short, \emph{$\sigma$-PD}\textemdash is Det($\{X\subseteq \mathds{R}: X \mbox{ is $\sigma$-projective}\}$).
\end{definition}

 We expect\textemdash and it is true\textemdash that:

\begin{propo}\label{prop:Menger proj is cpct}
$\sigma$-PD implies every Menger $\sigma$-projective set is $\sigma$-compact.
\end{propo}

This was proved for ``projective'' in \cite{TT}; the same proof works.

\begin{definition}\label{def pp}
Given a topological property $P$, a space $X$ is \emph{projectively $P$} if given any continuous function $f:X\rightarrow \mathds{R}$, $f(X)$ has property $P$.
\end{definition}

We will be interested in spaces which are projectively $\sigma$-compact, projectively $\sigma$-projective, or projectively countable. Some authors replace ``projectively'' by ``functionally''. Some replace $\mathds{R}$ by $[0, 1]$ or any separable metrizable space. These are all equivalent.

In analogy to one of the original definitions of $K$-analyticity, let us consider the class $\pmb{\mathscr{K}}$ of images of $\sigma$-projective sets under \emph{Upper Semi-Continuous Compact-Valued multifunctions}. Below is the definition of \emph{USCCV multifunctions}. We call members of $\pmb{\mathscr{K}}$ \emph{$K$-$\sigma$-projective}.

\begin{definition}
Given a set $X$ we will denote by $\mathcal{P} (X)$ its power set, the collection of all subsets of $X$. We say that a function $F:X\rightarrow \mathcal{P}(Y)$ (a \emph{multifunction}) is \emph{Upper Semi-Continuous} if given a closed set $C\subseteq Y$, its \emph{outer inverse} 
\[F^{-1}(C)=\left\{x\in X : F(x) \,\cap\, C\neq \emptyset\right\}\]
is closed.

Equivalently, given an open set $U\subseteq Y$, its \emph{inner inverse} 
\[F^{-1}_{in}(U)=\left\{x\in X : F(x)\subseteq U\right\}\]
is open.
\end{definition}

$F$ is \emph{compact-valued} if $F(x)$ is compact for all $x \in X$. $F$ is \emph{USCCV} if it is upper semi-continuous and compact-valued.

This may be a bit difficult to absorb, but more pleasantly we have,

\begin{definition}\label{def CIPP}
If the space $Y$ is the continuous image of a space $Z$ such that there is a perfect surjective function $f:Z\rightarrow X$, then we say that $Y$ is a \emph{CIPP (Continuous Image of a Perfect Preimage)} of $X$.
\end{definition}

Rogers and Jayne \cite{RJ80} proved:

\begin{repcorollary}{CIPP and USCCV}
Given spaces $X$ and $Y$, $Y$ is a USCCV image of $X$ if and only if it is a CIPP of $X$.
\end{repcorollary}

An easy corollary is:

\begin{corollary}[Folklore]\label{USCCV lin}
USCCV images of Lindel\"of spaces are Lindel\"of,
\end{corollary}

\begin{proof}
Continuous functions preserve Lindel\"of. Lindel\"of is a perfect property.
\end{proof}

Rogers and Jayne \cite{RJ80} proved that

\begin{reptheorem}{USCCV of irrationals}
    A space is $K$-analytic if and only if it is a USCCV image of $\mathbb{P}$.
\end{reptheorem}

Not so easy is

\begin{reptheorem}{U and C pp}
All members of $\pmb{\mathscr{K}}$ are projectively $\sigma$-projective.
\end{reptheorem}

We will postpone the proof until Section 5. 

$K$-analytic spaces are USCCV images of $\mathds{P}$; \emph{Lindel\"of $\Sigma$-spaces} are USCCV images of any separable metrizable space. Lindel\"of $\Sigma$-spaces show up in many contexts, which is why Tkachuk \cite{Tk10} calls them ``an omnipresent class''. In particular, the \emph{$K$-countably determined} spaces of \cite{RJ80} are precisely the Lindel\"of $\mathbf{\Sigma}$ spaces. $\pmb{\mathscr{K}}$ fits neatly between the class of $K$-analytic spaces and the class of Lindel\"of $\Sigma$-spaces. Assuming $\sigma$-PD, it is closer to the former class.

Most of the applications of $K$-analyticity use Lindel\"ofness. Thus $\pmb{\mathscr{K}}$ is a worthy generalization of the class of $K$-analytic sets, especially assuming $\sigma$-PD. However,
Menger spaces are trivially Lindel\"of, so in the context of getting Menger spaces to be
Hurewicz, much more inclusive classes than $\pmb{\mathscr{K}}$ can be considered. Indeed,

\begin{reptheorem}{Menger is Hurewicz}
In ZFC, Menger projectively analytic spaces are Hurewicz. $\sigma$-PD implies Menger projectively $\sigma$-projective spaces are Hurewicz.
\end{reptheorem}

We will then have:

\begin{repcorollary}{U Menger is Hur}
$\sigma$-PD implies Menger spaces in $\pmb{\mathscr{K}}$ are Hurewicz.
\end{repcorollary}

The $\sigma$-projective subsets of any Polish space are closed under continuous real-valued images, countable unions, and complements. Can we get a reasonable class $\pmb{\mathscr{L}}$  containing  all $K$-analytic spaces and closed under continuous images, countable unions and some sort of restricted complement? We discuss this question in Section \ref{K hier}.

\section{Definitions}\label{Section definitions}

All topological spaces that we work with will be completely regular.


\begin{definition}
Given a compactification $\mu X$ of $X$ the \emph{remainder of $X$ in $\mu X$} is $\mu X\setminus X$.
\end{definition}

\begin{definition}
We will denote by $\beta X$ the Stone-\v{C}ech compactification of $X$ and by $X^{\ast}$ its Stone-\v{C}ech remainder, $\beta X\setminus X$.
\end{definition}

\begin{definition}
A space is \emph{\v{C}ech-complete} if it is $G_{\delta}$ in its Stone-\v{C}ech compactification. Equivalently, if its Stone-\v{C}ech remainder is $\sigma$-compact.
\end{definition}

In addition to $\pmb{\mathscr{K}}$, we will consider two proper subclasses of it; 


\begin{definition}
Given a class of topological spaces $\Gamma$, $\pmb{\mathscr{C}}(\Gamma)$ is the class of all topological spaces that are continuous images of elements in $\Gamma$.
\end{definition}

\begin{definition}
Given a class of topological spaces $\Gamma$, $\pmb{\mathscr{P}}(\Gamma)$ is the class of all topological spaces that are perfect preimages of elements in $\Gamma$.
\end{definition}

\begin{definition}
Given a class of topological spaces $\Gamma$, $\pmb{\mathscr{K}}(\Gamma)$ is the class of all topological spaces that are USCCV images of elements in $\Gamma$.
\end{definition}

We call $\sigma\text{-} P$ the set of all $\sigma$-projective sets of reals (or subsets of $[0, 1]$, etc.). When referring to the classes $\pmb{\mathscr{K}}(\sigma\text{-} P)$, $\pmb{\mathscr{P}}(\sigma\text{-} P)$ or $\pmb{\mathscr{C}}(\sigma\text{-} P)$ we may also call them $\pmb{\mathscr{K}}$, $\pmb{\mathscr{P}}$ or $\pmb{\mathscr{C}}$ respectively.

\begin{definition}
A topological space $X$ is \emph{perfect} if every closed set of $X$ is a $G_{\delta}$. This is standard terminology. There is no relationship between perfect spaces and perfect functions.
\end{definition}

\section{A powerful tool: USCCV multifunctions}\label{USCCV Section}

In \cite{Why65} Whyburn refers to multifunctions (functions from a set to the power set of some set) as a useful approach to encompass a variety of topological results. Whyburn presents Upper Semi-Continuous Compact-Valued (USCCV) multifunctions as the right analogue of continuous functions in  multifunctions settings. A first important thing to notice about USCCV multifunctions is that every continuous function is (a single-valued) USCCV multifunction. Also that the outer (or inner) inverse of a multifunction is also a multifunction. Together with these, Whyburn showed that:

\begin{propo}{\emph{\cite{Why65}}}\label{Whyburn results}
\begin{enumerate}
 \item Images of compact sets under USCCV are compact (hence the composition of USCCV multifunctions is USCCV).
 \item A multifunction $F:X\rightarrow Y$ is closed and has compact inverse values for singletons if and only if $F^{-1}$ is USCCV.
\end{enumerate}
\end{propo}

Notice that, in 2), $F$ does not need to be Upper Semi-Continuous and that $F^{-1}$ is the outer inverse.

As mentioned in Section 1, $K$-analytic spaces can be defined in terms of USCCV multifunctions. We are interested in working with the Stone-\v{C}ech remainders of topological spaces in order to understand $\pmb{\mathscr{K}}$ and its variations better. We therefore work on generalizing some classic results by replacing ``continuous'' with ``USCCV'' and ``$\mathbf{\Sigma}^1_{\beta}$ ($\mathbf{\Pi}^1_{\beta}$)'' with ``$\pmb{\mathscr{K}}(\mathbf{\Sigma}^1_{\beta})$ ($\pmb{\mathscr{K}}(\mathbf{\Pi}^1_{\beta})$)''.

\begin{propo}\mbox{}
\begin{enumerate}
 \item Given a continuous function $f:X\rightarrow K$ with $K$ compact and $f(X)$ dense in $K$, there is a unique continuous function $\hat{f}:\beta X\rightarrow K$ extending $f$ such that $\hat{f}\restrict X=f$ and $\hat{f}(\beta X)=K$ (generalized in Theorem \ref{USCCV SC-comp}).
 \item Given a perfect map $f:X\rightarrow K$ with $K$ compact and $\hat{f}$ the extension given by the Stone-\v{C}ech compactification, we have that $\hat{f}(\beta X\setminus X)\subseteq K\setminus f(X)$ (generalized in Theorem \ref{USCCV SC-comp}).
 \item Perfect real-valued images of co-analytic sets are co-analytic (generalized in Corollary \ref{U perfect property} and Lemma \ref{lem:TallGeneralization}).
 \item Given a surjective continuous function $f:Y\rightarrow G$ with $Y$ Polish, $G$ separable metrizable and $X\subseteq G$ a $\mathbf{\Sigma}^1_{\beta}$ ($\mathbf{\Pi}^1_{\beta}$) set, then $f^{-1}(X)$ is $\mathbf{\Sigma}^1_{\beta}$ ($\mathbf{\Pi}^1_{\beta}$) (generalized in Theorem \ref{USCCV preimages U}).
\end{enumerate}
\end{propo}

These results are all in the literature. 1) and 2) can be found in \cite{Porter1988} and \cite{Eng89}. 3) is implicit in \cite{Tall20}. 4) is done for projective sets in \cite{Kec95} and the same proof works. 

To achieve the corresponding generalizations, it is useful to think that the co-domain of a multifunction is the space itself and not its power set. With this in mind, we have the following definitions:

\begin{definition}
We say that a multifunction $F:X\rightarrow Y$ is \emph{closed (open)} if for any closed (open) set $A\subseteq X$ we have that $F(A)=\bigcup_{a\in A}F(a)$ is closed (open) in $Y$.
\end{definition}

\begin{definition}
We say that $F$ is a \emph{perfect USCCV multifunction} if $F$ is a USCCV closed multifunction with compact point-inverses.
\end{definition}

\begin{propo}\label{Perfect inverse of perfect}
The outer inverse of a perfect USCCV multifunction is a perfect USCCV multifunction. 
\end{propo}

\begin{proof}
Let $F:X\rightarrow Y$ be a surjective perfect USCCV multifunction. Using Proposition \ref{Whyburn results}, point 2), we know that $F^{-1}$ is USCCV. Since $F$ is USCCV, given a closed set $C\subseteq Y$, 
\[\bigcup_{b\in C}F^{-1}(b)=\{x: (\exists b\in C) (b\in F(x))\}= \{x: F(x)\cap C\neq \emptyset\}=F^{-1}(C)\]
is closed.

Finally, given $x\in X$ 
\[(F^{-1})^{-1}(x)=(F^{-1})^{-1}(\{x\})=\{y: \{x\} \,\cap\, F^{-1}(y)\neq \emptyset\}=\]
\[\{y: x\in F^{-1}(y)\}=\{y: y\in F(x)\}=F(x).\]
Since $F$ is USCCV,  we have that $(F^{-1})^{-1}(x)$ is compact and $F^{-1}$ is a perfect USCCV multifunction.
\end{proof}

Rogers and Jayne proved:

\begin{propo}\label{USCCV graphs}
\emph{\cite{RJ80}} Given a USCCV multifunction $F:X\rightarrow A$, the graph of $F$ (the space \emph{$\mbox{graph}(F)=\bigcup_{x\in X}\{x\}\times F(X)$}) is a perfect preimage of $X$ (the perfect function being the projection to $X$) and the projection $\pi_{A}:\mbox{graph}(F)\rightarrow A$ is a continuous single-valued function.
\end{propo}

\begin{corollary}\label{CIPP and USCCV}
Given spaces $X$ and $Y$, $Y$ is a USCCV image of $X$ if and only if it is a CIPP of X.
\end{corollary}

\begin{proof}
If $Y$ is a CIPP of $X$ then, we know that every continuous function is USCCV and, by Proposition \ref{Whyburn results} point 2), the (outer) inverse multifunction of a perfect function is also USCCV. Furthermore, Proposition \ref{Whyburn results} point 1) tells us that the composition of USCCV multifunctions is USCCV. This shows that $Y$ is a USCCV image of $X$. If $Y$ is a USCCV image of $X$, then, by Proposition \ref{USCCV graphs}, $Y$ is the continuous image of $\mbox{graph}(F)$ which is a perfect preimage of $X$. So $Y$ is a CIPP of $X$.
\end{proof}

Although the definition of CIPP is easier to digest, the use of USCCV multifunctions is a better tool when dealing with compositions. In other words, it is not obvious that the CIPP of a CIPP of a space $X$ is a CIPP of $X$. Nevertheless, thanks to Whyburn's Proposition \ref{Whyburn results}, it is straightforward that the USCCV image of a USCCV image of $X$ is also a USCCV image of $X$.

\begin{definition}
Given multifunctions $F_i:X_i\rightarrow Y_i$ with $i\in I$, we say that $F:\prod_{i\in I} X_i\rightarrow \prod_{i\in I}Y_i$ defined as $F(f)=\prod_{i\in I}F_i(f(i))$ is the \emph{product multifunction} of $\{F_i: i\in I\}$.
\end{definition}

\begin{propo}\label{powerfully USCCV}
The product of USCCV multifunctions is USCCV.
\end{propo}

\begin{proof}
Let $F_i:X_i\rightarrow Y_i$, $i\in I$ be USCCV multifunctions and $F:\prod_{i\in I} X_i\rightarrow \prod_{i\in I}Y_i$ be such that $F(f)=\prod_{i\in I}F_i(f(i))$. It is immediate to see that $F(f)$ is compact whenever (and only when) for each $i\in I$ we have that $F_i(f(i))$ is compact.

Notice that, given $f\in \prod_{i\in I}X_i$, $F(f)\subseteq \prod_{i\in I}A_i\subseteq \prod_{i\in I}Y_i$ if and only if $F(f(i))\subseteq A_i$ for all $i\in I$. Equivalently, for every $i\in I$, $f(i)\in (F_i)^{-1}_{in}(A_i)$. This means that $F^{-1}_{in}(\prod_{i\in I}A_i)=\prod_{i\in I}(F_i)^{-1}_{in}(A_i)$.

Let $\prod_{i\in I}V_i\subseteq \prod_{i\in I}Y_i$ be a basic open set, i.e. each $V_i\subseteq Y_i$ is open and $V_i=Y_i$ for all but finitely many $i\in I$. Since $F_i$ is USCCV for every $i\in I$, $(F_i)^{-1}_{in}(V_i)$ is open. We also have that $(F_i)^{-1}_{in}(V_i)=(F_i)^{-1}_{in}(Y_i)=X_i$ for all but finitely many $i\in I$. Therefore, $F^{-1}_{in}(\prod_{i\in I}V_i)=\prod_{i\in I}(F_i)^{-1}_{in}(V_i)$ is (basic) open. This shows that $F$ is upper semi-continuous.
\end{proof}

\begin{definition}
A space $X$ is \emph{powerfully Lindel\"of} if $X^{\omega}$ is Lindel\"of.
\end{definition}

\begin{corollary}\label{USCCV powerfully  lin}
USCCV multifunctions preserve powerful Lindel\"ofness.
\end{corollary}

\begin{proof}
Assume that $X$ is powerfully Lindel\"of and let $F:X\rightarrow Y$ be a surjective USCCV multifunction. Using Proposition \ref{powerfully USCCV}, we know that $\prod_{i\in \omega}F: X^{\omega}\rightarrow Y^{\omega}$ is USCCV and that, since $X^{\omega}$ is Lindel\"of, we have that $Y^{\omega}=\prod_{i\in \omega}F(X^{\omega})$ is Lindel\"of.
\end{proof}

Now we focus on the following classic result.

\begin{propo}\label{SC perfect extension}
\emph{({\cite[3.7.16]{Eng89}})} The extension of a perfect map $g:Y\rightarrow X$ over the Stone-\v{C}ech compactification of both spaces, i.e. $\hat{g}:\beta Y\rightarrow \beta X$, is such that $\hat{g}(Y)=g(Y)=X$ and $\hat{g}(\beta Y\setminus Y)\subseteq \beta X \setminus X$. Furthermore, if $g$ is surjective, so is $\hat{g}$, which implies that $\hat{g}(\beta Y\setminus Y)= \beta X \setminus X$. Also, by \emph{\cite[3.7.6]{Eng89}}, for any $B \subseteq X$, $g\restrict g^{-1}(B)$ is perfect.
\end{propo}

A corollary of the above proposition is that given a map $g:X\rightarrow Y$, there exists a continuous extension $\hat{g}:\beta X\rightarrow \mu Y$ where $\mu Y$ is any compactification of $Y$. We call this the \emph{universal property of the Stone-\v{C}ech compactification}. Our next objective is to generalize the universal property of the Stone-\v{C}ech compactification to USCCV multifunctions:

\begin{propo}\label{USCCV perfect graphs}
Given a USCCV multifunction $F:X\rightarrow A$, and using the definitions and spaces of Proposition \ref{USCCV graphs}, we have that:

\begin{enumerate}
    \item If $F$ has compact point-inverses, then $\pi_A$ has compact point-inverses.
    \item If $F$ is a closed multifunction, then $\pi_A$ is a closed function.
\end{enumerate}
In particular, $F$ is a perfect multifunction if and only if $\pi_A$ is a perfect function.
\end{propo}

\begin{proof}\mbox{}
\begin{enumerate}
    \item Given $a\in F(X)$, we know that $F^{-1}(a)$ is compact, so $\pi_A^{-1}(a)=\{(x,a)\in \mbox{graph}(F): y\in F(x)\}=F^{-1}(a)\times\{a\}$ is also compact.
    
    \item Since we know that $\pi_A$ is continuous, we just need to show that it is closed. Let $C\subseteq \mbox{graph}(F): y\in F(x)\}$ be a closed set. Let $a\notin \pi_A(C)$ and $y\in F^{-1}(a)$. Let $O_y(a)\times U_y(a)\subseteq X\times A\setminus C$ be an open set such that $(y,a)\in O_y(a)\times U_y(a)$. In other words, $O_y(a)$ and $U_y(a)$ are such that for every $x\in O_y$, $\{x\}\times (F(x)\cap U_y)\cap C=\emptyset$. 
    
    Let $O(a)=\bigcup_{y\in F^{-1}(a)}O_y(a)$ and $U(a)=\bigcup_{y\in F^{-1}(a)}U_y(a)$. The open sets $O(a)$ and $U(a)$ still satisfy that for every $x\in O$, $\{x\}\times (F(x)\cap U)\cap C=\emptyset$. Furthermore, notice that $F^{-1}(a)\subseteq O(a)$, so for every $x\notin O(a) $, $a\notin F(x)$. Given that $F$ is closed, we have that $F(X\setminus O)$ is closed, $a\notin F(X\setminus O(a))$ and $V(a)=A\setminus F(X\setminus O(a))$ is open. This means that $a\in V(a)$. Furthermore, $a\in V(a)\cap U(a)$ and $V(a)\cap U(a)\cap \pi_A(C)=\emptyset$.
\end{enumerate}

To show the final equivalence, points 1) and 2) show that if a USCCV multifunction is perfect, then $\pi_A$ is also perfect. Now, if $\pi_A$ is perfect, using Propositions \ref{Perfect inverse of perfect} and \ref{USCCV graphs}, $\pi^{-1}_X$ is a perfect USCCV multifunction and $F=\pi_A\circ \pi^{-1}_X$ is also perfect.
\end{proof}

\begin{theorem}\label{USCCV SC-comp}
Given a USCCV multifunction $F:X\rightarrow K$ with $K$ compact and $F(X)$ dense in $K$, there is a USCCV multifunction $\hat{F}:\mu X\rightarrow K$ extending $F$ such that $\hat{F}\restrict X=F$ and $\hat{F}(\mu X)=K$, where $\mu X$ is any compactification of $X$. Furthermore, if $F$ is a perfect USCCV multifunction, then $\hat{F}(\mu X\setminus X)= K\setminus F(X)$.
\end{theorem}

\begin{proof}
Let $F:X\rightarrow K$ with $F$, $X$ and $K$ as described in the statement of the Theorem. Using Proposition \ref{USCCV graphs}, we know that $Z=\mbox{graph}(F)$ is a perfect preimage of $X$ and $F(X)\subseteq K$ is a continuous image of a projection from $Z$. Denote by $f:Z\rightarrow X$ and $g:Z\rightarrow K$ the surjective perfect function and the projection mentioned above.

Using the universal property of $\beta Z$ and Propositions \ref{Perfect inverse of perfect} and \ref{SC perfect extension}, we have that the continuous extension $\hat{f}:\beta Z \rightarrow \mu X$ is a perfect surjective function and that $\hat{f}^{-1}:\mu X\rightarrow \beta Z$ is a perfect surjective USCCV multifunction. Furthermore, since $f$ is perfect, $\hat{f}^{-1}(X)=Z$ and $\hat{f}^{-1}(\mu X \setminus X)=\beta Z\setminus Z$. For the continuous extenson of $g$, given that $K$ is compact and $g(Z)=F(X)$ is dense in $K$, $\hat{g}:\beta Z\rightarrow K$ is a surjective function. Define $\hat{F}: \mu X\rightarrow K$ as $\hat{F}=\hat{g}\circ \hat{f}^{-1}$.

Notice that $\hat{F}$ is a composition of two surjective USCCV multifunctions, so it is also surjective USCCV. On the other hand, since we know the exact definition of $Z$, we know that for each $x\in X$, 
\[\hat{F}(x)=\hat{g}\circ \hat{f}^{-1}(x)=g(\{x\}\times F(x))=F(x).\]
This shows that $\hat{F}$ is a USCCV extension of $F$.

Finally, assume that $F$ was perfect to begin with. Using the same definitions as above and Proposition \ref{USCCV perfect graphs}, we know that $g$ is perfect. Therefore, 
\[\hat{F}(\mu X\setminus X)=\hat{g}\circ \hat{f}^{-1}(\mu X\setminus X)=\hat{g}(\beta Z\setminus Z)=K\setminus g(Z)=K\setminus F(X).\]
\end{proof}

\section{Important old results and some new proofs}\label{section old results}

As seen in the last section, perfect functions behave in a nice way with respect to the Stone-\v{C}ech compactification and its remainder. Because of that, we can find multiple results in the literature about this special kind of function. In addition to Proposition \ref{SC perfect extension}, here we compile some other results that will be useful:

\begin{propo}\label{Cech complete prod}
\emph{\cite{Eng89}} The  countable product, countable intersection and the finite union of \v{C}ech-complete spaces are \v{C}ech-complete.
\end{propo}

\begin{propo}\emph{\cite{RJ80}}\label{K closure under countable}
$K$-analytic spaces are closed under countable unions, countable intersections, countable products, and continuous maps.
\end{propo}

\begin{propo}\label{perfect restriction}
\emph{\cite[3.7.6]{Eng89}} Given a perfect map $g:X\rightarrow Y$ and $B\subseteq Y$, the function $g\restrict g^{-1}(B):g^{-1}(B)\rightarrow B$ is perfect. 
\end{propo}

\begin{definition}
An space $X$ is \emph{nowhere locally compact (NLC)} if every compact subset of $X$ has empty interior.
\end{definition}

\begin{propo}\label{nwlc SC-Com}
\emph{\cite{vD92}} If a space $X$ is nowhere locally compact, then it and its Stone-\v{C}ech remainder $X^{\ast}$ have the same Stone-\v{C}ech compactification and $X^{\ast}$ is nowhere locally compact.
\end{propo}

\begin{propo}\label{Perfect property lemma}
\emph{\cite{Arh2013}} Let $P$ be a perfect property and $X$ a topological space. If the remainder of $X$ in some compactification has property $P$, then the remainder in any compactification has property $P$.
\end{propo}

For the analysis of the hierarchies and classes of our interest, the following results will be useful. We will give proofs of Proposition \ref{USCCV of irrationals} and its Corollary \ref{$K$-anal is perfect} that are simpler than the ones in \cite{RJ80}.

\begin{propo}\label{USCCV of irrationals}
\emph{\cite{RJ80}} Every USCCV image of $\mathds{P}$ is $K$-analytic.
\end{propo}

\begin{proof}
Let $F:\omega^{\omega}\rightarrow X$ be a surjective USCCV multifunction. Using Proposition \ref{USCCV graphs}, $\mbox{graph}(F)$ is a perfect preimage of $\omega^{\omega}$, a Lindel\"of \v{C}ech-complete space. Then, $\mbox{graph}(F)$ is  Lindel\"of \v{C}ech-complete. By the same proposition, $X$ is a continuous image of it and therefore $X$ is $K$-analytic.
\end{proof}

\begin{propo}\label{perfect K preimage}
\emph{\cite{RJ80}} The perfect preimage of a $K$-analytic space is $K$-analytic.
\end{propo}

\begin{corollary}\label{$K$-anal is perfect}
Being $K$-analytic is a perfect property.
\end{corollary}

\begin{proof}
Since $K$-analyticity is preserved by continuous maps, it is certainly preserved by perfect images. Proposition \ref{perfect K preimage} gives us the other part.
\end{proof}

\begin{propo}\label{K analytic metrazable}
\emph{\cite{RJ80}} Every $K$-analytic metrizable space is analytic.
\end{propo}

\begin{proof}
Let $F:\omega^{\omega}\rightarrow X$ be a surjective USCCV multifunction with $X$ a separable metrizable space. Using Proposition \ref{USCCV graphs}, $\mbox{graph}(F)\subseteq \omega^{\omega}\times X$ is a perfect preimage of a Lindel\"of \v{C}ech-complete space inside a separable metrizable space. Then it is a Polish space, since \v{C}ech-completeness is a perfect property \cite{Eng89}. Using Proposition \ref{USCCV graphs} again, $X$ is a continuous image of the Polish space $\mbox{graph}(F)$ and therefore $X$ is analytic.
\end{proof}

\begin{corollary}\label{K is PP}
\emph{\cite{RJ80}} $K$-analytic spaces are projectively $\sigma$-projective, indeed, projectively analytic.
\end{corollary}

As is standard, we can use whatever Polish space is convenient. For our purposes, since we will sometimes want to take a metrizable compactification of a separable metrizable space, it will be convenient to assume our analytic, co-analytic, projective, $\sigma$-projective sets are subsets of $[0,1]$, so their metrizable compactification is just their closure there.

\begin{definition}\label{def s-space}
A Tychonoff space $X$ is an \emph{$s$-space} if there exists a countable family $\mathcal{S}$ of open subsets of $\mu X$, a compactification of $X$, such that $X=\bigcup_{i\in I}\bigcap \mathcal{S}_i$ such that $\mathcal{S}_i\subseteq \mathcal{S}$ for all $i\in I$. 
\end{definition}

\begin{definition}\label{def p-space}
A space is a \emph{Lindel\"of $p$-space} if it is the perfect preimage of a separable metrizable space.
\end{definition}

\begin{definition}\label{def Lindelof Sigma}
A space is a \emph{Lindel\"of $\Sigma$-space} if it is the continuous image of a Lindel\"of $p$-space.
\end{definition}

\begin{definition}\label{def countable type}
A space $X$ has \emph{countable type} if for any compact subspace $K\subseteq X$ there exists $K'$, a compact subspace of $X$, such that $K\subseteq K'$ and there is a countable base for the open sets including $K'$.
\end{definition}

\begin{propo}
\emph{\cite{Arh2013}} Every Lindel\"of $p$-space is a Lindel\"of $s$-space.
\end{propo}

\begin{propo}\label{remainder of s}
\emph{\cite{Arh2013}} A space $X$ is an $s$-space if and only if any (some) remainder of $X$ is a Lindel\"of $\Sigma$-space.
\end{propo}

\begin{propo}\label{perfect s is p}
\emph{\cite{Arh2013}} Every perfect $s$-space is a $p$-space.
\end{propo}

\begin{propo}\label{remainder of countable type}
\emph{\cite{HI58}} A space is Lindel\"of if and only if its Stone-\v{C}ech remainder is of countable type.
\end{propo}

\section{Properties of $\pmb{\mathscr{K}}$}\label{Section U and C}

The main theorem of the next section will show that, under $\sigma$-PD, every Menger space in $\pmb{\mathscr{K}}$ is Hurewicz. In this section, we present several properties of $\pmb{\mathscr{K}}$. We start with unions, products and intersections.

\begin{theorem}\label{U and C properties I}
Given $\xi\in \omega_1$:
\begin{enumerate}
 
 \item $\pmb{\mathscr{K}}(\mathbf{\Pi}^1_{\xi})\cup \pmb{\mathscr{K}}(\mathbf{\Sigma}^1_{\xi})\subseteq\pmb{\mathscr{K}}(\mathbf{\Sigma}^1_{\xi+1})\cap \pmb{\mathscr{K}}(\mathbf{\Pi}^1_{\xi+1})$.
 
 \item The classes $\pmb{\mathscr{K}}(\mathbf{\Pi}^1_{\xi+1})$ are closed under countable unions, countable intersections.
 
 \item $\pmb{\mathscr{K}}(\mathbf{\Sigma}^1_{\xi
 +1})$ is closed under countable unions, countable intersections and countable products.

 \item For $\gamma\in \omega_1$ a limit ordinal,  $\pmb{\mathscr{K}}(\mathbf{\Sigma}^1_{\gamma})$ is closed under countable unions, finite intersections and finite products.
 
 \item For $\gamma\in \omega_1$ a limit ordinal, $\pmb{\mathscr{K}}(\mathbf{\Pi}^1_{\gamma})$  is closed under finite unions, countable intersections and countable products.
 
 \item Both $\pmb{\mathscr{K}}(\mathbf{\Pi}^1_{\xi})$ and $\pmb{\mathscr{K}}(\mathbf{\Sigma}^1_{\xi})$ are closed under USCCV multifunctions.
\end{enumerate}
\end{theorem}

\begin{proof}
Points 1) to 5) follow from the classic results of the $\sigma$-projective hierarchy and Proposition \ref{powerfully USCCV}.

Point 6) follows from their respective definitions and the fact that the composition of USCCV multifunctions is USCCV.
\end{proof}

\begin{corollary}
Every space in $\pmb{\mathscr{K}}$ is powerfully Lindel\"of.
\end{corollary}

\begin{corollary}
For every $\beta\in\omega_1$, $\pmb{\mathscr{K}}(\mathbf{\Sigma}^1_{\beta+1})=\pmb{\mathscr{K}}(\mathbf{\Pi}^1_{\beta})$.
\end{corollary}

\begin{proof}
Both results come from the facts that $\mathbf{\Pi}^1_{\beta}\subseteq \mathbf{\Sigma}^1_{\beta+1}$ and that, by definition, every member in $\mathbf{\Sigma}^1_{\beta+1}$ is the continuous image of a member in $\mathbf{\Pi}^1_{\beta}$.
\end{proof}

In light of the above result, we will only refer to the levels of $\pmb{\mathscr{K}}$ as $\pmb{\mathscr{K}}(\mathbf{\Sigma}^1_{\beta})$, with $\beta\in \omega_1$. The choice of $\mathbf{\Sigma}$ over $\mathbf{\Pi}$ is not arbitrary, since the $\pmb{\mathscr{K}}(\mathbf{\Pi}^1_{\beta})$ classes do not encompass $\pmb{\mathscr{K}}(\mathbf{\Sigma}^1_{\gamma})$ with $\gamma$ a limit ordinal.

\begin{corollary}\label{U perfect property}
For every $\beta\in\omega_1$, being a\, $\pmb{\mathscr{K}}(\mathbf{\Sigma}^1_{\beta})$ ($\pmb{\mathscr{K}}(\mathbf{\Pi}^1_{\beta})$) space is a perfect property.
\end{corollary}

\begin{proof}
Recall that if $f:A\rightarrow B$ is a perfect function, then both $f$ and $f^{-1}$ are USCCV multifunctions. By point (6) of Theorem \ref{U and C properties I}, every $\pmb{\mathscr{K}}(\mathbf{\Sigma}^1_{\beta})$ ($\pmb{\mathscr{K}}(\mathbf{\Pi}^1_{\beta})$) class is closed under USCCV images.
\end{proof}

It follows immediately that

\begin{corollary}
    Being in $\pmb{\mathscr{K}}$ is a perfect property.
\end{corollary}

\begin{lemma}
For every $\beta\in\omega_1$, $\pmb{\mathscr{K}}(\mathbf{\Sigma}^1_{\beta})$ is closed hereditary.
\end{lemma}

\begin{proof}
Assume that $B\in \pmb{\mathscr{K}}(\mathbf{\Sigma}^1_{\beta})$ and that $F:A\rightarrow B$ is a surjective USCCV multifunction with $A\in\mathbf{\Sigma}^1_{\beta}$. Also, let $C\subseteq B$ a closed set. On the one hand, we know that $F^{-1}(C)$ is closed in $A$ and, as shown by classical descriptive set theory, that implies that $F^{-1}(C)\in\mathbf{\Sigma}^1_{\beta}$. On the other hand, we can see $F$ as a function with co-domain $\beta B$. With this point of view, if we let $D$ be the (compact) closure of $C$ in $\beta B$, we have that $D\cap B=C$ and that, for every $a\in F^{-1}(C)$, $F(a)\cap D=F(a)\cap C$ is compact. Then the function $G: F^{-1}(C)\rightarrow C$ defined as $G(a)=F(a)\,\cap\, C$ is a USCCV surjective multifunction. This shows that $C\in \pmb{\mathscr{K}}(\mathbf{\Sigma}^1_{\beta})$.
\end{proof}

Now we can show that $\pmb{\mathscr{K}}$ is projectively $\sigma$-projective.

\begin{theorem}\label{U and C pp}
Let $Z\in \pmb{\mathscr{K}}$ and let $f: Z\rightarrow Y\subseteq \mathds{R}$ be a surjective continuous function. Then $Y$ is a $\sigma$-projective set.
\end{theorem}

\begin{proof}
Assume that $Z\in \pmb{\mathscr{K}}(\mathbf{\Pi}^1_{\gamma})$. Let $X$ be a $\mathbf{\Pi}^1_{\gamma}$ set and $G: X\rightarrow Z$ a USCCV surjective multifunction. Let $F=f\circ G$. Notice that $F: X\rightarrow Y$ is a USCCV surjective multifunction.

Without loss of generality, assume $X, Y\subseteq [0,1]$. This implies that $\overline{X}$ and $\overline{Y}$ are metric compactification of $X$ and $Y$. By Theorem \ref{USCCV SC-comp}, we can extend $F$ to a surjective multifunction $\hat{F}:\overline{X}\rightarrow \overline{Y}$.

Using Proposition \ref{USCCV graphs}, notice that $\mbox{graph}(\hat{F})\subseteq [0,1]^2$ is a separable metrizable space. Furthermore,  $\mbox{graph}(\hat{F})$ maps perfectly onto $\overline{X}$ and maps continuously onto $\overline{Y}$. By Theorem \ref{U and C properties I} point 6), $\pi_{\overline{X}}^{-1}(X)$ is $\mathbf{\Pi}^1_{\gamma}$. This means that $\pi_{\overline{Y}}(\pi_{\overline{X}}^{-1}(X))=Y$ is $\mathbf{\Sigma}^1_{\gamma+1}$.

The result for $\mathbf{\Sigma}^1_{\gamma+1}$ follows immediately from the $\mathbf{\Pi}^1_{\gamma}$ case and the proof for $\mathbf{\Sigma}^1_{\gamma}$ when $\gamma$ is a limit follows from the $\mathbf{\Sigma}^1_{\delta+1}$ results.
\end{proof}

\begin{corollary}\label{crl:separable metric in U}
If $Y\in \pmb{\mathscr{K}}(\mathbf{\Sigma}^1_{\gamma})$ is a metrizable space then $Y\in \mathbf{\Sigma}^1_{\gamma}$. 
\end{corollary}

\begin{proof}
$Y$ is separable because Lindel\"of metrizable spaces are separable. The proof is analogous to Theorem \ref{U and C pp} but assuming that $Y\subseteq [0,1]^{\omega}$. (Separable metrizable spaces need not embed in $[0, 1]$ but do embed in $[0, 1]^{\omega}$.)
\end{proof}

If one does not like pretending that $[0, 1]$ and $[0, 1]^{\omega}$ are the same, one can use the fact (see e.g. \cite{Kec95}) that every separable metrizable space is a perfect image of a subspace of the Cantor space and hence of $\mathds{R}$. (Incidentally, in \cite{Tall20} the second author followed \cite{RJ80} a little too closely, unnecessarily privileging $\mathds{R}^{\omega}$ over $\mathds{R}$ in the table of equivalents on page 11.) 

\begin{corollary}\label{proper hier I}
For every $\alpha<\beta \in \omega_1$, there is a space $B$ such that $B$ is $\pmb{\mathscr{K}}(\mathbf{\Sigma}^1_{\beta})$ but is not $\pmb{\mathscr{K}}(\mathbf{\Sigma}^1_{\alpha})$.
\end{corollary}

\begin{proof}
It is well-known that the $\sigma$-projective sets form a hierarchy of length $\omega_1$ \cite{MoschDST}. Let $B\in \mathbf{\Sigma}^1_{\beta}\setminus \mathbf{\Sigma}^1_{\alpha}$. We know that $B$ is $\pmb{\mathscr{K}}(\mathbf{\Sigma}^1_{\xi})$. Since $B\notin \mathbf{\Sigma}^1_{\alpha}$, by \ref{crl:separable metric in U} $B\notin \pmb{\mathscr{K}}(\mathbf{\Sigma}^1_{\alpha})$.
\end{proof}

Although we presented a simpler proof for Proposition \ref{K analytic metrazable}, the technique presented in Rogers and Jayne's Theorem 5.5.1 of \cite{RJ80} gives other characterizations of metrizability for $K$-analytic spaces. These can be generalized to the whole $\pmb{\mathscr{K}}$ hierarchy. Analyzing their proof carefully, it can be generalized to prove the following assertion:

\begin{theorem}\label{RJ generalized}
Given a class of topological spaces $\Gamma$ and a pair of topological properties $P$ and $Q$ such that 
\begin{enumerate}
 \item Every member of $\Gamma$ with property $Q$ is a USCCV image of a space with property $P$.
 \item For $X\in \Gamma$ with property $Q$, $X\times X$ is Lindel\"of and $X$ has a $G_{\delta}$ diagonal.
 \item $\Gamma$, property $Q$ and property $P$ are closed-hereditary.
 \item Every countable product of spaces with property $P$ has property $P$.
 \item Finite unions of spaces with property $P$ have property $P$.
\end{enumerate}
Then every member of $\Gamma$ with property $Q$ is a continuous image of a space with property $P$.
\end{theorem}

\begin{proof}
Let $P$, $Q$ and $\Gamma$ be as described in the Theorem. Furthermore, let $X\in \Gamma$ have property $Q$. We shall use $\triangle$ to denote the diagonal of $X$, $\triangle = \{\langle x, x\rangle\,:\,x \in X\}$. Then since $X\times X$ is Lindel\"of and $X$ has a $G_{\delta}$ diagonal, $(X\times X)\setminus \triangle$ is an $F_{\sigma}$ and, therefore, also Lindel\"of. 

As a consequence of $(X\times X)\setminus \triangle$ being Lindel\"of, there exist open $G_i^{0}, G_i^{1}\subseteq X$, for $i\in \omega$, such that  $(X\times X)\setminus \triangle=\bigcup_{i\in\omega}G_i^{0}\times G_i^{1}$. For each $\langle j,i\rangle \in (\{0, 1\}\times\omega)$, let $A_i^j=X\setminus G_i^j$. This sequence of pairs of sets has some interesting properties. First, since all of these sets are closed, by hypothesis, each $A^j_i\in \Gamma$ and has property $Q$. Also, since $G^0_i\times G^1_i \,\cap\, \triangle= \emptyset$, $G^0_i \,\cap\, G^1_i=\emptyset$ and $A^0_i \,\cup\, A^1_i=X$. Finally, given $x,y\in X$ such that $x\neq y$, there is an $i_{\langle x,y \rangle}$ such that $\langle x,y \rangle\in G^0_{i_{\langle x,y \rangle}}\times G^1_{i_{\langle x,y \rangle}}$. Notice that $x\in G^0_{i_{\langle x,y \rangle}}$, $x\notin G^1_{i_{\langle x,y \rangle}}$, $y\in G^1_{i_{\langle x,y \rangle}}$ and $y\notin G^1_{i_{\langle x,y \rangle}}$.

By the first point in the hypothesis, for each $\langle j, i \rangle\in \{0, 1\}\times \omega$, there is a surjective USCCV multifunction $L^j_i: B^j_i\rightarrow A^j_i$ such that $B^j_i$ has property $P$.  By the fifth and fourth points in our hypothesis, we have that
\[\prod_{i\in \omega}\left((B^{0}_i\times\{0\}) \,\cup\, (B^{1}_i\times\{1\})\right)\]
has property $P$. Define, for each $i\in \omega$, $\alpha(i)=(x_i^{\alpha}, b_i^{\alpha})$ and $N_i=(B^{0}_i\times\{0\}) \,\cup\, (B^{1}_i\times\{1\})$.
 
Since each $L^j_i$ is USCCV and surjective over $A^j_i$, Proposition \ref{powerfully USCCV} implies that the multifunction
\[M:\prod_{i\in\omega}N_i\rightarrow X^{\omega}\]
with $M(\alpha)=\prod_{i\in \omega}L^{j_i^{\alpha}}_i(B_i^{\alpha})$ is surjective USCCV. Notice that the subspace 
\[\vec{X}=\{\xi\in X^{\omega}: (\forall i,n\in \omega) (\xi(i)=\xi(n))\}\]
inside $X^{\omega}$ is homeomorphic to $X$ with its induced topology. Furthermore, $\vec{X}$ is closed in $X^{\omega}$.
To see this, define the open set $D^{x}_{i,j}=X$ when $i\neq j$ and $D^{x}_{i,i}=X\setminus \{x\}$, for each $i,j\in\omega$ and $x\in X$. We can see that $X^{\omega}\setminus \vec{X}=\bigcup_{x\in X}\bigcup_{i\in \omega}\prod_{j\in \omega}D^{x}_{i,j}$. This implies that $C=M^{-1}(\vec{X})$ is closed in $\prod_{i\in\omega}N_i$, so it has property $P$.

To finish the proof, we will show that $M\restrict C: C\rightarrow \vec{X}$ is single-valued. This will show that $X$ is homeomorphic to a continuous image of a space with property $P$. An element of $\vec{X}$ will be denoted by $\vec{x}$.


Assume that there is an $\alpha\in C$ and $x,y\in X$ such that $\vec{x}, \vec{y}\in M(\alpha)$ and $x\neq y$. In particular, we know that 
\[x,y \in L^{j_{i_{\langle x,y \rangle}}^{\alpha}}_{i_{\langle x,y \rangle}}(B_{i_{\langle x,y \rangle}}^{\alpha})\subseteq A^{j_{i_{\langle x,y \rangle}}^{\alpha}}_{i_{\langle x,y \rangle}}.\]
But, as mentioned when we defined $i_{\langle x,y \rangle}$, it is not possible for both $x$ and $y$ to be in $A^{j_{i_{\langle x,y \rangle}}^{\alpha}}_{i_{\langle x,y \rangle}}$. Therefore, for each $\alpha\in C$, $|M(\alpha) \,\cap\, X^{\omega}|=1$. This shows that  $M\restrict C$ is a continuous surjective single-valued function.
\end{proof}

As we can see, when following the \cite{RJ80} proof one ends up with a single-valued function to the desired space, $X$, from the product of countably many spaces that were, originally, the domain of countably many USCCV multifunctions to subspaces of $X$.

Here is an easy consequence of Theorem \ref{RJ generalized}:

\begin{corollary}
    Every Lindel\"of $\mathbf{\Sigma}$-space with a $G_\delta$ diagonal is a continuous image of a separable metrizable space.
\end{corollary}

\begin{proof}
    Let $\Gamma$ be the class of all topological spaces. Let property $Q$ be ``being a Lindel\"of $\mathbf{\Sigma}$-space with a $G_\delta$ diagonal''. Let property $P$ be ``being a separable metrizable space''.
\end{proof}

Jayne and Rogers use $\Gamma$ as the $K$-analytic spaces, $Q$ as ``separable metrizable space'' and $P$ as ``analytic set'' (they use $\omega^{\omega}$, but analytic is enough). With this strategy, they showed that for $K$-analytic spaces ``analytic space'' is equivalent to ``spaces such that $(X\times X)\setminus \triangle$ is Lindel\"of''. This partially generalizes if we replace ``analytic'' by ``$\mathbf{\Sigma}^1_{\beta}$'':

\begin{theorem}\label{thm:cosmic diagonal}
Given $X\in \pmb{\mathscr{K}}(\mathbf{\Sigma}^1_{\beta})$, if $X$ has a $G_{\delta}$ diagonal then $X$ is a continuous image of a $\mathbf{\Sigma}^1_{\beta}$ set.
\end{theorem}

\begin{proof}
Since $X$ is powerfully Lindel\"of and $X$ has a $G_{\delta}$ diagonal, then $(X\times X)\setminus \triangle$ is Lindel\"of. Using Theorem \ref{RJ generalized} with $\Gamma=\pmb{\mathscr{K}}(\mathbf{\Sigma}^1_{\beta})$, $P$ the property ``being a $\mathbf{\Sigma}^1_{\beta}$ set'' and $Q$ the (trivial) property ``$(X\times X)\setminus \triangle$ is Lindel\"of''; we have that every space $X$ in $\pmb{\mathscr{K}}(\mathbf{\Sigma}^1_{\beta})$ such that $(X\times X)\setminus \triangle$ is Lindel\"of is a continuous image of a $\mathbf{\Sigma}^1_{\beta}$ set. This finishes the proof.
\end{proof}

Theorem \ref{RJ generalized} gives us a new proof of a result that we proved earlier.

\begin{repcorollary}{crl:separable metric in U}
If $X\in \pmb{\mathscr{K}}(\mathbf{\Sigma}^1_{\beta})$ is a metrizable space, then $X \in \mathbf{\Sigma}^1_{\beta}$.
\end{repcorollary}

\begin{proof}
As before, we note $X$ is separable. Using Theorem \ref{RJ generalized} with $\Gamma=\pmb{\mathscr{K}}(\mathbf{\Sigma}^1_{\beta})$, $P$ the property ``being a $\mathbf{\Sigma}^1_{\beta}$ set'' and $Q$ the property ``$X$ is a separable metrizable space''; we have that every space $X$ in $\pmb{\mathscr{K}}(\mathbf{\Sigma}^1_{\beta})$ that is metrizable is a continuous image of a $\mathbf{\Sigma}^1_{\beta}$ set. This implies that $X$ is a $\mathbf{\Sigma}^1_{\beta}$ set.
\end{proof}

Note that point 2) in Theorem \ref{RJ generalized} creates a path for constructing single-valued functions from USCCV multifunctions even when not working in the separable metrizable space context (but using a property $Q$ of interest).

Finally, we would like to know whether the levels of the hierarchy are preserved by preimages.

\begin{theorem}\label{USCCV preimages U}
Let $Y$ be a $\pmb{\mathscr{K}}(\mathbf{\Sigma}^1_{\xi})$ space, $X\subseteq G$ a  $\pmb{\mathscr{K}}(\mathbf{\Sigma}^1_{\xi})$ space and $F:Y\rightarrow G$ a surjective USCCV multifunction such that $\hat{F}\circ \hat{F}^{-1}(X)=X$ for $\hat{F}:\mu Y\rightarrow \beta G$ a USCCV extension of $F$. Then $F^{-1}(X)$ is $\pmb{\mathscr{K}}(\mathbf{\Sigma}^1_{\xi})$. 
\end{theorem}

\begin{proof}
Let $Y$, $G$, $X$ and $F$ be as described in the theorem. Consider the function given by Theorem \ref{USCCV SC-comp}, $\hat{F}:\mu Y\rightarrow \beta G$. Since $\hat{F}$ is a perfect USCCV multifunction, by Proposition \ref{Perfect inverse of perfect} $\hat{F}^{-1}$ is also perfect. Since $\hat{F}\circ \hat{F}^{-1}(X)=X$, $\hat{F}^{-1}\restrict X: X\rightarrow \hat{F}^{-1}(X)$ is a surjective perfect function. By Theorem \ref{U and C properties I} point 6), $\hat{F}^{-1}(X)$ is $\pmb{\mathscr{K}}(\mathbf{\Sigma}^1_{\xi})$. Since $Y$ is $\pmb{\mathscr{K}}(\mathbf{\Sigma}^1_{\xi})$, using a different point of Theorem \ref{U and C properties I} we have, 
\[F^{-1}(X)=\hat{F}^{-1}(X) \,\cap\, Y\]
is $\pmb{\mathscr{K}}(\mathbf{\Sigma}^1_{\xi})$.
\end{proof}

\begin{theorem}\label{continuous preimage U}
Let $X$ be a $\pmb{\mathscr{K}}(\mathbf{\Sigma}^1_{\xi})$ space. Let $f:Y\rightarrow X$ be a surjective continuous function. Then $Y$ is $\pmb{\mathscr{K}}(\mathbf{\Sigma}^1_{\xi})$. 
\end{theorem}

\begin{proof}
Let $\hat{f}:\beta Y\rightarrow \beta X$ the continuous extension of $f$ given by the Stone-\v{C}ech compactification. We know that $\hat{f}\circ \hat{f}^{-1}(X)=X$. Using Theorem \ref{USCCV preimages U}, we conclude that $Y$ is $\pmb{\mathscr{K}}(\mathbf{\Sigma}^1_{\xi})$.
\end{proof}

\section{Applications to Selection Principles}\label{Section SP}

Now we shall consider Selection Principles for $\pmb{\mathscr{K}}$ and related classes.

We don't need the following alternative definition of ``Hurewicz'', but it may be of interest to those unfamiliar with Hurewicz spaces.

\begin{propo}
\emph{\cite{Tall2011}} A topological space $X$ is Hurewicz if and only if, given a \v{C}ech-complete $G$ such that $X\subseteq G$, there exists a $\sigma$-compact space $F$ such that $X\subseteq F\subseteq G$.
\end{propo}

\begin{theorem}\label{Menger is Hurewicz}
$\sigma$-PD implies every Menger projectively $\sigma$-projective space is Hurewicz. In ZFC, Menger projectively analytic spaces are Hurewicz.
\end{theorem}

\begin{proof}
This follows from \begin{propo}{\emph{\cite{Tall11Hu}}}\label{Lin menger are hur}
Lindel\"of projectively $\sigma$-compact spaces are Hurewicz.\vspace{-0.2cm}
\end{propo}
\end{proof}

\begin{corollary}\label{U Menger is Hur}
$\sigma$-PD implies Menger spaces in $\pmb{\mathscr{K}}$ are Hurewicz.
\end{corollary}

As mentioned in the introduction, this conclusion cannot easily be improved since there is a $K$-analytic space that is Hurewicz but not $\sigma$-compact. Also, $\pmb{\mathscr{K}}$ is not the same as the class of all projectively $\sigma$-projective spaces.

\begin{repexamp}{Okunev space}
A Hurewicz $K$-analytic space that is not $\sigma$-compact: Okunev's space.
\end{repexamp}

\begin{repexamp}{Peng L-space}
There are projectively $\sigma$-projective spaces that are not in $\pmb{\mathscr{K}}$: Peng's L-space.
\end{repexamp}

Nevertheless, demanding a little more structure on the space lets us get a stronger conclusion. The following results come from \cite{KKP} or are direct generalizations of results of Rogers and Jayne in \cite{RJBorel}. For the following proofs it will be more convenient to use $\omega^{\omega}$ rather than $\mathds{R}$ or $[0,1]$. 

\begin{propo}\label{perfect special not sigma subspace}
\emph{\cite{KKP}} For every perfect Lindel\"of space $X$ that is not $\sigma$-compact, there exists a non-empty closed subset $N\subseteq X$ such that for every open set $U$ of $X$, $U\cap N$ is not $\sigma$-compact.
\end{propo}







\begin{lemma}\label{perfect special sequence of closed sets}
For every perfect Lindel\"of space $X$ that is not $\sigma$-compact, there exist $L_n$, $n\in\omega$, pairwise disjoint closed sets that are not $\sigma$-compact and such that $\bigcup_{n\in \omega}L_n$ is closed.
\end{lemma}

\begin{proof}
Let $X$ be a perfect topological Lindel\"of space that is not $\sigma$-compact and $N$ be the closed subspace obtained from Proposition \ref{perfect special not sigma subspace}. We will work with the induced topology on $N$ from this point onward.

Since $N$ is not compact but is Lindel\"of, it is not countably compact, so there is a sequence of countably many decreasing closed sets, $F_n\subseteq N$, such that $\bigcap_{n\in \omega}F_n=\emptyset$ and $F_0=N$. Since $X$ and $N$ are both perfect, for each $n\in \omega$ there exist countably many decreasing open sets $O^m_n\subseteq N$, such that $\bigcap_{m\in\omega}O^m_n=F_n$ and such that $O_n^{m+1}\subseteq \overline{O_n^{m+1}}\subseteq O_n^{m}$.

Let $E^{(i)}=\bigcap_{n,m\leq i} \overline{O_n^m}$ and $U^{(i)}=\bigcap_{n,m\leq i}O^m_n$. For each $i\in \omega$, $E^{(i)}$ is closed and $U^{(i)}$ is open. Furthermore, $U^{(i+1)}\subseteq E^{(i+1)}\subseteq U^{(i)}$ and
\[\bigcap_{i\in \omega}U^{(i)}=\bigcap_{i\in \omega}E^{(i)}= \bigcap_{n\in \omega}F_n=\emptyset.\]

By definition, $\overline{O^m_0}=O^m_0=E^{(0)}=U^{(0)}=N$. Therefore, there must be an $i_1$ such that $N\setminus U^{(i_1)}$ is not $\sigma$-compact. Otherwise, 
\[N=N\setminus \emptyset= N\setminus \bigcap_{i\in\omega}U^{(i)}= \bigcup_{i\in \omega}N\setminus U^{(i)}\]
would be $\sigma$-compact.

Let $i_0=-1$ and for all $n\in \omega$, $n\geq 1$, define $i_n$ such that $E^{(i_n+1)}\setminus U^{(i_{n+1})}$ is not $\sigma$-compact. This number always exists. To see this, we can use the same argument we used for $N$ with an induction. To simplify notation, we will write $E^{(n)}$ instead of $E^{(i_n)}$ (we substitute for $E^{(i)}$ the subsequence $E^{(i_n)}$).

Now we can define our desired closed sets. For each $n\in \omega$, let $L_n=E^{(i_n+1)}\setminus U^{(i_{n+1})}$. We know that each $L_n$ is closed and not $\sigma$-compact. Furthermore, since $E^{(i+1)}\subseteq U^{(i)}$, we know that they are pairwise disjoint. To finish the proof, let $x\in \overline{\bigcup_{n\in \omega} L_n}$. Since $\bigcap_{i\in \omega} E^{(i)}=\emptyset$ and the sets $E^{(i)}$ are decreasing, there is an $m_x\in \omega$ such that for all $n\geq m_x$, $x\notin E^{(n)}\subseteq E^{(m_x)}$. Let $G_x$ be an open neighborhood of $x$ that is disjoint from $E^{(m_x)}$. Notice that for all $n\geq m_x$, $G_x\,\cap\,L_n=\emptyset$. Then, 
\[G_x\cap \bigcup_{n\in\omega}L_n=G_x \cap \bigcup_{n< m_x}L_n.\]
This means that $x\in \overline{\bigcup_{n< m_x}L_n}=\bigcup_{n< m_x}L_n$. So $x\in \bigcup_{n\in \omega}L_n$ and $\bigcup_{n\in \omega}L_n$ is closed.
\end{proof}

We need some new notation before continuing.

\begin{definition}
Given $N\subseteq Y$, $F: A\rightarrow Y$ a surjective multifunction, and $D\subseteq A$, let $D_N = \{a \in A\,:\,F(a) \cap N \neq \emptyset\}$. Let $F_N: D_N \to N$ be defined by $F_N(a) = F(a) \cap N$.
\end{definition}

Recall that for $s\in \omega^{<\omega}$
\[[s]=\{a\in \omega^{\omega} : s\subseteq a\}\]
is (basic) open and for $T\subseteq \omega^{<\omega}$
\[[T]=\{a\in \omega^{\omega}: (\forall n\in \omega) (a\restrict n\in T)\}\]
denotes the set of all branches of $T$ and is closed.

\begin{definition}
Given $s \in \omega^{<\omega}$, $T\subseteq \omega^{<\omega}$, and $A\subseteq \omega^{\omega}$ we define $[s]_{A}=[s]\cap A$ and $[T]_{A}=[T]\cap A$.
\end{definition}

\begin{definition}
Given $T\subseteq \omega^{<\omega}$ and $s\in \omega^{<\omega}$ we define
$(T)_s$ to be the set of all successor nodes of $s$ that are in $T$.
\end{definition}

The sets $[s]_{A}$ and $[T]_{A}$ define the induced topology on $A$ (see chapter I.2.B of \cite{Kec95}).

\begin{definition}
We say that a family $\{N_s: s\in \omega^{<\omega}\}$ is a \emph{disjoint Souslin scheme} if
\begin{enumerate}
    \item If $s\subseteq t$, $N_t\subseteq N_s$.
    \item If $s\not\subseteq t$ and $t\not\subseteq s$, $N_s\cap N_t=\emptyset$.
\end{enumerate}

Furthermore, we say that a disjoint Souslin scheme is \emph{closed} if
\begin{enumerate}\setcounter{enumi}{2}
    \item For all $s\in \omega^{<\omega}$, $N_s\neq \emptyset$ and $N_s$ is closed.
    \item For each $n\in \omega$, $\bigcup_{s\in \omega^{n}}N_s$ is closed.
\end{enumerate}
\end{definition}

\begin{definition}
Given $A\subseteq \omega^{\omega}$, the \emph{generating tree of $A$} is \[T_{A}=\{s\in \omega^{<\omega}:(\exists b\in A) (\exists n\in \omega) (b\restrict n=s)\}.\]

Given $U\subseteq \omega^{<\omega}$, the \emph{tree generated by $U$} is
\[T_{U}=\{s\in \omega^{<\omega}:(\exists t\in U) (\exists n< |t|) (t\restrict n=s)\}.\]
\end{definition}

Notice that if $A\subseteq B\subseteq \omega^{\omega}$ then $[T_A]_B=A$ if and only if $A$ is closed in $B$.

\begin{lemma}\label{Perfect Souslin}
Assume that $X$ is a perfect Lindel\"of $\Sigma$-space. If $X$ is not $\sigma$-compact, then there is a disjoint closed Souslin scheme such that for every $s\in \omega^{<\omega}$, $N_s$ is a non-empty closed non-$\sigma$-compact subset of $N$.
\end{lemma}

\begin{proof}
Let $X$ be a Lindel\"of $\Sigma$-space, let $B\subseteq \omega^{\omega}$ and let $F:B\rightarrow X$ a surjective USCCV multifunction. Let $N$ be as in Proposition \ref{perfect special not sigma subspace}. For each $s\in \omega^{<\omega}$ we will recursively define $N_s\subseteq N$. First, let $N_{\emptyset}=N$. Now, assume that we have defined $N_s$. Let $L_i$, $i\in\omega$ be the countably many closed sets given by Lemma \ref{perfect special sequence of closed sets} applied to $N_s$. For each $i\in \omega$, let $N_{s^{\frown}i}=L_i$.

To finish the proof, we sketch the induction argument used to show that for each $n\in \omega$, $\bigcup_{s\in \omega^{n}}N_s$ is closed. First, $N_{\emptyset}=N$. Assuming the result for $n$, if $|s|=n$, then, by the properties of $L_i$, $\bigcup_{i\in\omega}N_{s^{\frown}i}\subseteq N_s$ is closed. Then, using the same properties, $\bigcup_{s\in \omega^{n}}\bigcup_{i\in\omega}N_{s^{\frown}i}=\bigcup_{t\in\omega^{n+1}}N_{t} $ is closed.
\end{proof}

It is important to remark that, although we were able to create inside $X$ the sets $N_s$ for each $s\in \omega^{<\omega}$, that does not mean that for every $r\in \omega^{\omega}$ we have that $\bigcap_{s\subseteq r}N_s\neq \emptyset$. This last statement would create a copy of $\omega^{\omega}$ inside $X$, which is not true in general.

\begin{lemma}\label{Nice F_N}
Assume that $X$ is a perfect Lindel\"of $\Sigma$-space and let $B\subseteq \omega^{\omega}$ be such that $X$ is the USCCV image of $B$ by $F$. If $X$ is not $\sigma$-compact, then there is a non-$\sigma$-compact closed subset $N\subseteq X$ and a closed set $A\subseteq B$ such that $F_N\restrict A$ is surjective and the USCCV image of every non-empty open subset of $A$ is non-$\sigma$-compact.
\end{lemma}

\begin{proof}
Let $X$ be a Lindel\"of $\Sigma$-space, let $B\subseteq \omega^{\omega}$ and let $F:B\rightarrow X$ be a surjective USCCV multifunction. Let $N\subseteq X$ be the closed non-$\sigma$-compact subspace coming from Proposition \ref{perfect special not sigma subspace}. Since USCCV multifunctions preserve compactness,  $A=F^{-1}(N)\subseteq B$ is a closed subset of $B$ that is not $\sigma$-compact. Furthermore, there are only countably many basic open sets $O\subseteq A$ such that $F(O)\cap N$ is $\sigma$-compact. If $\mathcal{O}$ is the collection of all basic open subsets of $A$ such that their images are $\sigma$-compact in $N$, then $A\setminus \bigcup \mathcal{O}$ and $N\setminus F(\bigcup \mathcal{O})$ are non-empty and non-$\sigma$ compact. With this, we can assume that for any open subset of $A$, the image of that open subset is not $\sigma$-compact. Finally, let $F_N:A\rightarrow N$ be such that $F_N(b)=F(b)\cap N$ for all $b\in A$.
\end{proof}

\begin{lemma}\label{Perfect Souslin and USCCV}
Assume that $X$ is a perfect Lindel\"of $\Sigma$-space and let $B\subseteq \omega^{\omega}$ be such that $X$ is the USCCV image of $B$. If $X$ is not $\sigma$-compact, then there is a non-$\sigma$-compact closed subset $A$ of $B$, a $g:\omega^{<\omega}\rightarrow T_{A}$, and a disjoint closed Souslin scheme such that:
\begin{enumerate}
    \item For every $s\in \omega^{<\omega}$, $(T_{g(\omega^{<\omega})})_{g(s)}$ is infinite and $\overline{F_N([g(s)]_{A})}\cap N_s$ is a closed non-empty non-$\sigma$-compact subset of $N$.
    \item $[T_{g(\omega^{<\omega})}]_{A}$ is a closed non-empty non-$\sigma$-compact subset of $A$.
    \item If $b\in [T_{g(\omega^{<\omega})}]_{A}$ and $g(s)\subseteq b$ then $F_N(b)\cap N_s\neq \emptyset$.
\end{enumerate}
\end{lemma}

\begin{proof}
Let $X$ be a Lindel\"of $\Sigma$-space, let $B\subseteq \omega^{\omega}$ and let $F:B\rightarrow X$ be a surjective USCCV multifunction. Furthermore, let $N$ and $A$ be as in Lemma \ref{Nice F_N}.

We know that all compact sets of $\omega^{\omega}$ are the branches of finite branching trees \cite{Kec95}. Since no open subset of $A$ is $\sigma$-compact, we know that for any $s\in A$ there is a $t\in A$, with $s\subseteq t$, such that $(T_A)_{t}=\{i\in \omega: t^{\frown}i\in A\}$ is infinite. 

For each $s\in \omega^{<\omega}$ we will recursively define $N_s\subseteq N$ in a similar way to what we did in the proof of Lemma \ref{Perfect Souslin}. As before, let $N_{\emptyset}=N$. Now, assume that we have defined $N_s$ and $g(s)\in A$ for $s\in \omega^{<\omega}$. Let $L_i$ be the countably many closed sets given by Lemma \ref{perfect special sequence of closed sets} applied to $ \overline{F_N([g(s)]_{A})}\cap N_s$. We will find $\{t_i: i\in \omega\}\subseteq (T_A)_{g(s)}$ such that if $i\neq j$, $t_i\neq t_j$, and such that $\overline{F_N([t_i]_{A})}\,\cap\, L_i$ is non-empty and non-$\sigma$-compact. 

For each $i\in \omega$,
\[L_i=L_i\cap F_N([g(s)])=L_i\cap \bigcup_{t\in (T_A)_{g(s)}}F_N([t]_{A})=\bigcup_{t\in (T_A)_{g(s)}}F_N([t]_{A})\cap L_i.\]
If $F_N([t]_{A})\cap L_i$ were $\sigma$-compact for all $t\in (T_A)_{g(s)}$, then $L_i$ would be $\sigma$-compact. On the other hand, if there are only finitely many $t\in (T_A)_{g(s)}$ such that $\overline{F_N([t_i]_{A})}\,\cap\,\bigcup_{i\in \omega}L_i$ is non-empty (and non-$\sigma$-compact), then $F_N([g(s)])\setminus \bigcup_{i\in \omega}L_i=\bigcup_{t\in (T_A)_{g(s)}}F_N([t]_{A})\setminus \bigcup_{i\in \omega}L_i$ is non-empty and non-$\sigma$-compact. So we can use Lemma \ref{perfect special sequence of closed sets} again. In this manner we can find the desired $\{t_i: i\in \omega\}\subseteq (T_A)_{g(s)}$. Furthermore, notice that we will still have the property of the union of closed sets being closed, since the proof only requires that the $L_i$ form a countable collection of closed sets with empty intersection.

For each $i\in \omega$, let $g(s^{\frown}i)\in T_A$ be such that $t_i\subseteq g(s^{\frown}i)$ and $(T_A)_{g(s^{\frown}i)}$ is infinite. Finally, define $N_{s^{\frown}i}\subseteq \overline{F_N([t_i]_{A})}\cap L_i$ to be the closed set given by Proposition \ref{perfect special not sigma subspace}. Remember that $N_{s^{\frown}i}$ is obtained from $\overline{F_N([t_i]_{A})}\cap L_i$ by subtracting a $\sigma$-compact open set, and $F_N([s]_{A})$ is non-$\sigma$-compact for all $s\in A$. With this, elements from the proof of Lemma \ref{Perfect Souslin}, and the above construction, we can verify that the sets $N_s$ form a disjoint closed Souslin scheme that satisfies points 1) and 3). To finish the proof, we will show point 2) by contradiction.

Assume that $[T_{g(\omega^{<\omega})}]_{A}$ is $\sigma$-compact. Then there is an $f\in \omega^{\omega}$ such that for all $b\in [T_{g(\omega^{<\omega})}]_{A}$ there is an $n_{b}\in \omega^{\omega}$ such that $b(m)<f(m)$ for all $m\geq n_b$. Take $a\in [T_{g(\omega^{<\omega})}]_{A}$ such that $a(n_b)<f(m)$; by construction, there is an $s_a\in \omega^{<\omega}$ such that $a\subseteq s_a$ and $(T_{g(\omega^{<\omega})})_{s_a}$ is infinite. This means that we can find a $b_f\in [T_{g(\omega^{<\omega})}]_{A}$ such that $b_f\restrict |s_a|+1\in (T_{g(\omega^{<\omega})})_{s_a}$ and $b_f(|s_n|)>f(|s_n|)$. This contradicts the assumption of $[T_{g(\omega^{<\omega})}]_{A}$ being $\sigma$-compact.
\end{proof}

\begin{theorem}\label{perfect S-Lind closed preimage}
Assume that $X$ is a perfect Lindel\"of $\Sigma$-space and let $B\subseteq \omega^{\omega}$ be such that $X$ is the USCCV image of $B$. If $X$ is not $\sigma$-compact, then there is a non-$\sigma$-compact closed subset of $X$ that is the perfect preimage of a closed subset of $B$.
\end{theorem}

\begin{proof}
Let $X$ be a Lindel\"of $\Sigma$-space, let $B\subseteq \omega^{\omega}$ and let $F:B\rightarrow X$ be a surjective USCCV multifunction. Let $N$, $g$, $A$ and $N_s$ for each $s\in \omega^{<\omega}$ be the sets and functions coming from Lemma \ref{Perfect Souslin and USCCV}. 

For each $b\in [T_{g(\omega^{<\omega})}]_{A}$, let $H(b)=\bigcap_{g(s)\subseteq b}N_s\cap F_N(b)$. Since $N_s$ is closed, $\{N_s: g(s)\subseteq b\}$ is nested, and $F_N(b)$ is compact, we know that $H(b)$ is a non-empty compact set. Also, notice that 
\[N'=\{H(b): b\in [T_{g(\omega^{<\omega})}]_{A}\}=\bigcap_{n\in \omega}\bigcup_{s\in\omega^{n}}N_{s} \] is closed. Then, to finish the proof, we need to show that 
\[H: [T_{g(\omega^{<\omega})}]_{A}\rightarrow N'\]
is a USCCV closed multifunction and that $H^{-1}$ is single-valued. With this, we will have that $H^{-1}$ is a perfect function.

First, if $a,b\in [T_{g(\omega^{<\omega})}]_{A}$ and $a\neq b$, we have that $H(a)\cap H(b)=\emptyset$. This shows that $H^{-1}$ is single-valued. Also, if $J\subseteq [T_{g(\omega^{<\omega})}]_{A}$ is closed, then $J=[V]_{[T_{g(\omega^{<\omega})}]_{A}}$ for some subtree $V$ of $\omega^{<\omega}$. By Lemma \ref{Perfect Souslin}, we know that $\bigcup_{s\in \omega^{n}\cap T}N_{s}$ is closed for all $n\in \omega$. Therefore, 
\[H(J)=H([V]_{[T_{g(\omega^{<\omega})}]_{A}})=\bigcap_{n\in\omega}\bigcup_{s\in \omega^{n}\cap V}N_{s}\]
is closed. This shows that $H$ is a closed multifunction.

We know that $H(b)$ is compact for each $b\in [T_{g(\omega^{<\omega})}]_{A}$. Let $U\subseteq [T_{g(\omega^{<\omega})}]_{A}$ be open and let $H(b)\subseteq U$ for some $b\in [T_{g(\omega^{<\omega})}]_{A}$. Notice that, since $\bigcap_{g(s)\subseteq b}N_s\cap F_N(b) = H(b)$, there must be a $s_{b, U}\in \omega^{<\omega}$ such that $g(s_{b,U})\subseteq b$ and $N_{s_{b,U}}\subseteq U$. Then, 
\[b\in [g(s_{b,U})]_{[T_{g(\omega^{<\omega})}]_{A}}\subseteq H^{-1}_{in}(N_{s_{b,U}})\subseteq H^{-1}_{in}(U).\]
This shows that $H^{-1}_{in}(U)$ is open and so $H$ is USCCV.
\end{proof}

\cite{TT} proved that $K$-analytic perfect Menger spaces are $\sigma$-compact; we will extend that result here.

\begin{theorem}\label{perfect HD}
(Hurewicz Dichotomy) $\sigma$-PD implies that every perfect space in $\pmb{\mathscr{K}}$ that is not $\sigma$-compact has a closed subset that is a perfect preimage of $\omega^{\omega}$.
\end{theorem}

\begin{proof}
Let $X$ be a perfect space in $\pmb{\mathscr{K}}$ that is not $\sigma$-compact, $B\subseteq \omega^{\omega}$ a $\sigma$-projective set, and $F:B\rightarrow X$ a USCCV surjective multifunction. By Theorem \ref{perfect S-Lind closed preimage}, we can find a non-$\sigma$-compact closed set $N'\subseteq X$ and $A\subseteq B$ such that there is a surjective perfect function $f:N'\rightarrow A$.

Since $\sigma$-projectivity is a closed hereditary property, we know that $A$ is $\sigma$-projective. Then, using $\sigma$-PD, since $A$ is not $\sigma$-compact there is a closed subset of $A$ homeomorphic to $\omega^{\omega}$ \cite{Ke77}. Call this closed subset $W$. Notice that $f^{-1}(W)$ is a closed subset of $N'$ and $X$ and that $f\restrict f^{-1}(W)$ is perfect. Therefore, $f^{-1}(W)\subseteq X$ is a closed subset that is a perfect preimage of $\omega^{\omega}$.
\end{proof}

\begin{corollary}\label{perfect menger}
$\sigma$-PD implies that every perfect Menger space in $\pmb{\mathscr{K}}$ is $\sigma$-compact.
\end{corollary}

\begin{proof}
By Theorem \ref{perfect S-Lind closed preimage}, if $X\in \pmb{\mathscr{K}}$ is perfect and not $\sigma$-compact, then $X$ has a closed set, say $C$, that is the perfect preimage of a non-$\sigma$-compact $\sigma$-projective space, say $D$. If $X$ is Menger, then $C$ and $D$ are also Menger. But $D$ contradicts $\sigma$-PD (since that implies that every Menger $\sigma$-projective set is $\sigma$-compact).
\end{proof}

In \cite{KKP} the authors prove a version of Hurewicz Dichotomy for hereditarily Lindel\"of $K$-analytic subspaces of hereditarily paracompact spaces. They call their Theorem 3.13 the \emph{Calbrix-Hurewicz Theorem}. We expect we can prove a $\sigma$-PD version of that theorem, but have not tried.

\begin{definition}\label{def cosmic}
A space $X$ is \emph{cosmic} if it is the continuous image of a separable metrizable space.
\end{definition}

Arhangel’ski\u{\i}-Calbrix \cite{AC1} proved

\begin{theorem}
    Every Menger cosmic analytic space is $\sigma$-compact. 
\end{theorem}

We can now see that the key fact is that cosmic spaces are hereditarily Lindel\"of and hence perfect. We saw earlier (Theorem \ref{thm:cosmic diagonal}) that Lindel\"of $\mathbf{\Sigma}$-spaces with $G_\delta$ diagonals are cosmic. Thus

\begin{corollary}
    $K$-analytic Menger spaces with $G_\delta$ diagonals are $\sigma$-compact \cite{Tall20} and $\sigma$-PD implies Menger members of $\pmb{\mathscr{K}}$ with $G_\delta$ diagonals are $\sigma$-compact.
\end{corollary}

\begin{proof}
    By Theorem \ref{perfect HD} and Corollary \ref{perfect menger}.
\end{proof}

\section{Variations on $\pmb{\mathscr{K}}$}\label{K hier}
In this section, we discuss subclasses of $\pmb{\mathscr{K}}$ and possible enlargements of it. For example, let $\pmb{\mathscr{C}}$ be the class of continuous images of $\sigma$-projective sets. Not surprisingly, $\pmb{\mathscr{C}}$ is a proper subclass of $\pmb{\mathscr{K}}$---see Example \ref{In U not C}. We could also consider $\pmb{\mathscr{P}}$; the class of perfect preimages of $\sigma$-projective sets. Also not surprisingly, $\pmb{\mathscr{P}}$ is a proper subclass of $\pmb{\mathscr{K}}$.

\begin{theorem}
    $\sigma$-PD implies Menger members of $\pmb{\mathscr{C}}$ are $\sigma$-compact.
\end{theorem}

\begin{proof}
    Members of $\pmb{\mathscr{C}}$ are cosmic, hence perfect, so \ref{perfect menger} applies.
\end{proof}

\begin{theorem}\label{thm:Menger perfect preimages}
Menger perfect preimages of analytic sets are $\sigma$-compact. $\sigma$-PD implies Menger members of $\pmb{\mathscr{P}}$ are $\sigma$-compact. 
\end{theorem}

\begin{proof}
Let $X$ be a Menger space, $B\subseteq [0,1]$ a $\sigma$-projective set and $f:X\rightarrow B$ a perfect surjective function. Since being Menger is preserved by continuous functions, we have that $f(X)=B$ is Menger. By $\sigma$-PD, $B$ is $\sigma$-compact. Therefore, its perfect preimage, $f^{-1}(B)=X$, is also $\sigma$-compact. The analytic case is analogous.
\end{proof}

Okunev's space (see Example \ref{Okunev space}) is $K$-analytic---and hence is in $\pmb{\mathscr{K}}$---but is not in $\pmb{\mathscr{P}}$ since it is Menger but not $\sigma$-compact. 

In generalizing classical descriptive set theory beyond the real line or separable metrizable spaces, a crucial obstacle is how to generalize the natural idea that the complement in $\mathds{R}$  or $[0,1]$ of a nice set is another nice set, worthy of being included in our hierarchy of nice sets. When we are no longer working within a fixed space such as $\mathds{R}$, taking a complement suddenly catapults us onto the shaky ground of classes.
One way of dealing with this classic problem is to make the relatively mild assumption that there are arbitrarily large inaccessible cardinals $\kappa$ and take for $X \in V_{\kappa}$ its complement in $V_{\kappa}$.  A more plausible approach to dealing with this problem is to take the complement of a space $X$ to be the Stone-\v{C}ech remainder $X^{\ast}$. It is convenient to have $\beta X^{\ast} = \beta X$; this will happen for nowhere locally compact (NLC) spaces, i.e. spaces such that no point has a compact neighbourhood \cite{vD92}. 
This is not such an onerous restriction: for any $X$, $X \times \mathds{Q}$ and $X\times \mathds{P}$ are NLC. By counting ultrafilters, one can prove that $|\beta X| \leq 2^{2^{|X|}}$. For $\kappa$ inaccessible, we then have that $V_{\kappa}$ is closed under Stone-\v{C}ech remainders, i.e. for any space $X$ in $V_{\kappa}$, $X^{\ast}$ is in $V_{\kappa}$, since $V_{\kappa}$ is a model of ZFC and thus satisfies Power Set and Separation axioms. Similarly, $V_{\kappa}$ is closed under continuous images, since they do not increase cardinality and $V_{\kappa}$ satisfies Replacement. 

The assumption that there is an inaccessible cardinal is much weaker than the large cardinal assumption needed to imply $\sigma$-PD, and is exactly the consistency strength of what is needed to get $\sigma$-projective Menger sets to be $\sigma$-compact \cite{TTT21}. If one wants an axiomatic approach, rather than working in a particular $V_{\kappa}$, one can assume there are arbitrarily large inaccessible cardinals.
Then the Stone-\v{C}ech remainder and all continuous images of a space will be sets rather than classes. Again, the assumption that there are arbitrarily large inaccessible cardinals is a much weaker assumption than is required to get $\sigma$-PD. Making this inaccessible assumption, we can then attempt to form an $\omega_1$-length hierarchy by closing under Stone-\v{C}ech remainders, continuous images, and countable unions, possibly restricting ourselves to NLC spaces. Does this make sense? Are spaces in this hierarchy projectively $\sigma$-projective, assuming $\sigma$-PD, and hence Menger ones Hurewicz? We made this claim in a seminar in January, 2022 at the Fields Institute, but that was premature and we withdraw it. There are difficulties.

Notice that our otherwise well-behaved class $\pmb{\mathscr{K}}$ is NOT closed under Stone-\v{C}ech remainder, even for nowhere locally compact spaces:

\begin{definition}
A space is \emph{co-$K$-analytic ($K$-$\mathbf{\Pi}^1_1$)} if its Stone-\v{C}ech remainder is a $K$-analytic space.
\end{definition}

A different reasonable definition of a space $X$ being co-$K$-analytic is that $X = Y^*$, for some $K$-analytic $Y$. By Proposition \ref{nwlc SC-Com}, these two definitions coincide for NLC spaces. 

\begin{repexamp}{A co-$K$-analytic not K}
For every cardinal $\kappa > \aleph_0$, if $D$ is the discrete space of size $\kappa$, then $D\times \omega^{\omega}$ is a \v{C}ech-complete NLC space that is not Lindel\"of but is co-$K$-analytic.
\end{repexamp}

\begin{repexamp}{Okunev space}
A $K$-analytic space with non-Lindel\"of Stone-\v{C}ech remainder: Okunev's space.
\end{repexamp}

If we try to close the $K$-analytic spaces under Stone-\v{C}ech remainders and continuous images, we quickly get every space!

\begin{theorem}\label{S-compact remainder}
Every topological space is the continuous image of an NLC space with $\sigma$-compact Stone-\v{C}ech remainder.
\end{theorem}

\begin{proof}
Let $X$ be a topological space of size $\kappa$ and let $D$ be the discrete space of size $\kappa$. The space $D\times \omega^{\omega}$ is a \v{C}ech-complete NLC space. This means that its Stone-\v{C}ech remainder is NLC $\sigma$-compact. We know that $X$ is a continuous image of $D$ and that $D$ is a continuous image of $D\times \omega^{\omega}$.
\end{proof}

Is there a more modest approach? Perhaps we should confine ourselves to Lindel\"of spaces, which is a very plausible restriction, since Lindel\"ofness is used in most $K$-analytic applications. We thus have the following definitions:

\begin{definition}
A topological space is \emph{$L$-$\mathbf{\Sigma}^1_{1}$} if it is $K$-analytic. A topological space is \emph{$L$-$\mathbf{\Pi}^1_{1}$} if it is a Lindel\"of space and its Stone-\v{C}ech remainder is an $L$-$\mathbf{\Sigma}^1_{1}$\,. In general, a topological space is \emph{$L$-$\mathbf{\Sigma}^1_{\xi+1}$}, $1\leq \xi<\omega_1$ if it is the continuous image of an $L$-$\mathbf{\Pi}^1_{\xi}$. A topological space is \emph{$L$-$\mathbf{\Pi}^1_{\xi}$}, $1\leq \xi<\omega_1$ if it is Lindel\"of and its Stone-\v{C}ech remainder is an $L$-$\mathbf{\Sigma}^1_{\xi}$.

For $\alpha<\omega_1$ a limit ordinal, a topological space $X$ is \emph{$L$-$\mathbf{\Sigma}^1_{\alpha}$} if there exists $Y_i\in L\text{-}\mathbf{\Sigma}^1_{\xi_i}$, $i<\omega$, $\xi_i<\alpha$ such that $X=\bigcup_{i\in \omega} Y_i$.

We will refer to $ L\text{-}\mathbf{\Sigma}^1_{\omega_1}$ as the class $\mathcal{L}$.
\end{definition}

Theorem \ref{S-compact remainder} shows we need to add something like this Lindel\"of restriction. Note that, as far as the conjecture that ``definable'' Menger spaces are $\sigma$-compact (or Hurewicz) is concerned, this is not so crucial, since Menger spaces are Lindel\"of, but if we want to generalize other properties of $K$-analytic spaces, this is a problem. The remainders of $K$-analytic spaces \textemdash even NLC ones \textemdash need not be Lindel\"of. 

We conjecture that restricting ourselves to members of $\mathcal{L}$ gives us a nicer class to work with, but we have still been unable to prove that such spaces are projectively $\sigma$-projective. Posing a concrete problem:

\begin{question}
Does $\sigma$-PD (actually, just $Det(\Pi^1_1)$) imply that Lindel\"of co-$K$-analytic spaces are projectively $\sigma$-projective?
\end{question}

Should we resign ourselves to adding extra conditions that are not necessarily satisfied by $K$-analytic spaces? One obvious try is to require countable type. Important examples to contend with are \emph{Okunev's space} and its remainder.

\begin{repexamp}{Okunev space}\label{eg:Okunev2}
A co-$K$-analytic NLC space, $(O \times \mathds{Q})^{\ast}$, that is not Lindel\"of. $O$ is Okunev's space.
\end{repexamp}

Okunev's space is an example of a $K$-analytic space that is not of countable type. If we require countable type, we will have remainders of Lindel\"of spaces Lindel\"of. But continuous images of even NLC $K$-analytic spaces of countable type need not be of countable type. For example, both $O$ and $O \times \mathds{Q}$ are images of spaces of countable type (see Example \ref{Okunev space} below).

There is one case in which we can add a not-too-onerous condition to a co-$K$-analytic space and get Menger such spaces to be $\sigma$-compact.

\begin{theorem}\label{co-k perfect Menger are sigma}
$Det(\Pi^1_1)$ implies perfect co-$K$-analytic Menger spaces are $\sigma$-compact.
\end{theorem}

\begin{proof}
Let $X$ be a perfect co-$K$-analytic Menger space. By definition, the Stone-\v{C}ech remainder of $X$ is $K$-analytic, hence, a Lindel\"of $\Sigma$-space. But Proposition \ref{remainder of s}, $X$ is an $s$-space, but Proposition \ref{perfect s is p} tells us that perfect $s$-spaces are $p$-spaces. In \cite{Tall20} the second author proved:

\begin{propo}\label{prop:Det coKp sigma}
    $Det(\Pi^1_1)$ implies Menger co-$K$-analytic $p$-spaces are $\sigma$-compact.
\end{propo}
\end{proof}

However, we have been unable to extend Theorem \ref{co-k perfect Menger are sigma}:

\begin{question}
Are perfect Menger continuous images of co-$K$-analytic spaces $\sigma$-compact?
\end{question}

Not every Menger $p$-space is $\sigma$-compact---consider a Menger non-$\sigma$-compact subset of $\mathds{R}$. But what about Menger co-$K$-analytic $p$-spaces? We answered that in Proposition \ref{prop:Det coKp sigma} above. In fact,

\begin{definition}
    We will say $X$ is \emph{co-$\pmb{\mathscr{K}}$} if $X^* \in \pmb{\mathscr{K}}$. We similarly define $X$ being \emph{co-$\pmb{\mathscr{P}}$}.
\end{definition}   

\begin{theorem}
    $\sigma$-PD implies Menger co-$\pmb{\mathscr{K}}$ $p$-spaces are $\sigma$-compact.
\end{theorem}

\begin{proof}
    We first prove a generalization of \cite[Lemma 1.9]{Tall20}:

    \begin{lemma}\label{lem:TallGeneralization}
        Perfect images of elements of co-$\pmb{\mathscr{K}}$ are elements of co-$\pmb{\mathscr{K}}$.
    \end{lemma}

    \begin{proof}
        Let $f$ be a perfect map mapping an $X \in \pmb{\mathscr{K}}$ onto $Y$. Extend $f$ to $\widehat{f}$ mapping $\beta X$ onto $\beta Y$. Then by Proposition \ref{SC perfect extension} $\widehat{f}(X^*) = Y^*$ and $\widehat{f}\,\text{\scalebox{1}[1.3]{$\upharpoonright$}}{\widehat{f}^{-1}(Y^*)} = \widehat{f}\,\text{\scalebox{1}[1.3]{$\upharpoonright$}}{X^*}$ is perfect. But $X^* \in \pmb{\mathscr{K}}$ and so $Y^* \in \pmb{\mathscr{K}}$.
    \end{proof}

    Now to prove the theorem, let $X$ be Menger co-$\pmb{\mathscr{K}}$. $X$ is a $p$-space so there is a perfect function $f$ mapping $X$ onto a separable metrizable space and hence into $[0, 1]^\omega$. By the Lemma, $f(X)$ is co-$\pmb{\mathscr{K}}$. Then $(f(X))^* \in \pmb{\mathscr{K}}$. Extend $f$ to $\widehat{f}$ mapping $\beta X$ onto the closure of $f(X)$ in $[0, 1]^{\omega}$. Being in $\pmb{\mathscr{K}}$ is a perfect property, so by Proposition \ref{Perfect property lemma}, the remainder of $f(X)$ in $[0, 1]^{\omega}$, i.e., $[0, 1]^{\omega}\setminus f(X)$ is in $\pmb{\mathscr{K}}$. But then $[0, 1]^{\omega}\setminus f(X)$ is $\sigma$-projective, so $f(X)$ is $\sigma$-projective. $f(X)$ is also Menger, so by $\sigma$-PD it is $\sigma$-compact, but $f$ is perfect, so $X$ is $\sigma$-compact.
\end{proof}

\begin{corollary}
    Menger co-$\pmb{\mathscr{P}}$ spaces are $\sigma$-compact.
\end{corollary}

\begin{proof}
    They are co-$\pmb{\mathscr{K}}$ $p$-spaces.
\end{proof}

\section{Examples and counterexamples}\label{Section Examples}

\begin{examp}\label{In U not C}
\emph{\textbf{A compact space in $\pmb{\mathscr{K}}$ not in $\pmb{\mathscr{C}}$.}}

Given that all the $\sigma$-projective sets are inside separable metrizable spaces, they are, at most, of size $\mathfrak{c}$. Since given $X\in\pmb{\mathscr{C}}$  there exists $B\subseteq \mathds{R}$ and $f:B\rightarrow X$ a continuous surjective function, we have that \[|X|\leq |B|\leq |\mathds{R}|= \mathfrak{c}\]

On the other hand, every compact space is a USCCV image of any $\sigma$-projective set. In particular, any ordinal $\omega_{\alpha}+1$ of cardinality strictly bigger than $\mathfrak{c}$ (with the order topology) is a Hausdorff compact space. These spaces are not in $\pmb{\mathscr{C}}$.

For an example of a space in $\pmb{\mathscr{K}}$ which is not in $\pmb{\mathscr{P}}$, see Okunev's space in \ref{Okunev space}. It is $K$-analytic but not of countable type, so its remainder is not Lindel\"of $p$, so it is not in $\pmb{\mathscr{P}}$.
\end{examp}







\begin{definition}
We say that an space is \emph{$k$-discrete} if it is the disjoint union of open compact subspaces.
\end{definition}

Clearly,

\begin{lemma}
A space is $k$-discrete if and only if it is the perfect preimage of a discrete space.
\end{lemma}

\begin{examp}\label{discrete remainder}
\textbf{\emph{A $K$-analytic space with $k$-discrete Stone-\v{C}ech remainder.}}

Let $D$ be an uncountable discrete space, $C$ its one-point compactification. Name the new point $\ast$. Notice that the space $[0,1]\times C$ is compact and let $X=\left((0,1)\times D \right)\cup \left([0,1]\times\{\ast\}\right)$.

Notice that $X$ is not compact since the intersection of the sequence of decreasing closed sets $(0,1/n)\times\{d\}$ is empty. Nevertheless, $X$ is $K$-analytic (actually, $\sigma$-compact) since \[X=\left((0,1)\times D \right)\cup \left([0,1]\times\{\ast\}\right)=\left((0,1)\times C \right)\cup \left([0,1]\times\{\ast\}\right)=\]
\[=\left(\bigcup_{n\in \omega}\left[1/n, 1-1/n\right]\times C\right) \cup \left([0,1]\times\{\ast\}\right).\]

On the other hand, $[0,1]\,\times\, C$ is compact and $X$ is dense in it. Furthermore, $[0,1]\times (C \setminus X)=\{0,1\}\times D$, a discrete space homeomorphic to $D$. Thus $X$ is a $K$-analytic space such that, in one of its compactifications, the remainder is homeomorphic to $D$.

By the universal property of the Stone-\v{C}ech compactification, there is a function $f:\beta X \rightarrow [0,1]\times C$ such that $f \restrict X^{\ast}$ is perfect and $f(X^{\ast})=\left([0,1]\times C\right) \setminus X$ which is homeomorphic to $D$. Therefore $X^{\ast}$ is a perfect preimage of the uncountable discrete space $D$.
\end{examp}

\begin{examp}\label{A co-$K$-analytic not K}
\textbf{\emph{A \v{C}ech-complete NLC space that is not Lindel\"of.}}




Let $D$ be the discrete space of size $\kappa$ for some $\kappa>\omega$. Since $D$ is locally compact, we can define its one-point compactification $\alpha D$. Using the universal property of the Stone-\v{C}ech compactification, we can extend the identity function to a surjective function $i:\beta D\rightarrow \alpha D$ with $i(\beta D\setminus D)=\alpha D\setminus D$. Since $\alpha D\setminus D$ is a single point, $\beta D\setminus D=i^{-1}(\alpha D\setminus D)$ is compact (a closed set in a compact space). This shows that $D$ is \v{C}ech-complete since its Stone-\v{C}ech remainder is a compact space.

By Proposition \ref{Cech complete prod}, $D\times \omega^{\omega}$ is a \v{C}ech-complete space. It is an NLC space that is not Lindel\"of since $D$ is uncountable. Therefore, it is not $K$-analytic, but its Stone-\v{C}ech remainder is NLC $K$-analytic since the remainder is $\sigma$-compact.
\end{examp}

\begin{definition}
A topological space $X$ is \emph{Rothberger} if given any countable sequence of open covers of $X$ $\langle\mathcal{U}_1, \mathcal{U}_2,...., \mathcal{U}_n,...\rangle $ there are $V_{i}\in \mathcal{U}_i$ such that $\{V_{i}:i\in \omega\}$ is also an open cover of $X$.
\end{definition}

\begin{examp}\label{Okunev space}
\textbf{\emph{A $K$-analytic space with a non-Lindel\"of remainder.}}

Let $A(\mathds{P})$ be the Alexandroff duplicate of the irrational numbers $\mathds{P}$. In other words, $A(\mathds{P})$ is the set $\mathds{P}\times \{0,1\}$ with the topology generated by $\{(r,1)\}$ for all $r\in \mathds{P}$ and $(U\times\{0\})\,\cup\, (\mathds{P}\setminus A)\times \{1\}$ where $U$ is open in the usual topology of $\mathds{P}$ and $A\subseteq \mathds{P}$ is finite. We will first show that this space is of countable type. Given a compact set $K\subseteq A(\mathds{P})$, we know that $K\,\cap\, (\mathds{P}\times \{1\})$ is compact, open and finite (so it has countable character). On the other hand, $K\cap (\mathds{P}\times \{0\})$ is a compact subset of $\mathds{P}$ which also has countable character. Therefore, $K$ has countable character in $A(\mathds{P})$.

\textbf{Okunev's space} $X$ is the quotient space $A(\mathds{P})/(\mathds{P}\times\{0\})$. This is the topological space obtained by collapsing the non-discrete copy of the irrationals to a single point. Okunev’s space is a $K$-analytic (see e.g. \emph{\cite{Tall20}}) – hence Lindel\"of \textemdash  space which is not of countable type, since it is $K$-analytic, Rothberger, and not $\sigma$-compact (see \emph{\cite{Tall20}}), but 
\begin{propo}
\emph{\cite{Tall20}} $K$-analytic Rothberger spaces of countable type are $\sigma$-compact.
\end{propo}
Since $X \times \mathds{Q}$ is also $K$-analytic, Rothberger, and not $\sigma$-compact, it is not of countable type, but also is nowhere locally compact.

Using Proposition \ref{remainder of countable type}, we know that $(X \times \mathds{Q})^{\ast}$ has a $K$-analytic remainder but is not Lindel\"of.
\end{examp}

\begin{examp}\label{Peng L-space}
\textbf{\emph{A Hereditarily Lindel{\"o}f projectively countable space with non-Lindel{\"o}f square}}

Moore \emph{\cite{Moore2005}} has an example of a hereditarily Lindel\"of space $X$ such that $X^2$ is not Lindel\"of (proved in \emph{\cite{Peng15}}), but $X$ is projectively countable, hence---in particular---projectively $\sigma$-projective. Notice that this space is not in $\pmb{\mathscr{K}}$ because $X^2$ is not Lindel\"of.
\end{examp}

\section{Future projections and questions}\label{Section future}

Although we have concentrated on the $\sigma$-projective sets of reals, most of our results extend to those sets of reals that are in $L(\mathds{R})$. Somewhat larger large cardinals than needed for $\sigma$-PD yield determinacy for $\mathcal{P}(\mathds{R}) \cap L(\mathds{R})$. All the results about Menger implying $\sigma$-compact or Hurewicz go through unchanged. One point that is not completely obvious, but is true, is that $\mathcal{P}(\mathds{R}) \cap L(\mathds{R})$ is closed under countable intersections, countable unions and continuous real-valued images. The point is that each of these can be coded by a real. That is clear for the first two; for the third, it is because continuous real-valued functions on the reals are determined by their values on rational arguments, so again this can be coded by a real.

If one is satisfied with consistency rather than direct implication, it should be noted that all our ``Menger implies $\sigma$-compact'' results, both for the $\sigma$-projective sets and for $\mathcal{P}(\mathds{R}) \cap L(\mathds{R})$, are equiconsistent with an inaccessible cardinal. For discussion, see \cite{TTT21}.

\bibliographystyle{alpha}
\bibliography{NTG}

\end{document}